\documentclass[12pt]{article}
\usepackage[english]{babel}
\usepackage{graphicx}
\usepackage{framed}
\usepackage[normalem]{ulem}
\usepackage{amsmath}
\usepackage{amsthm}
\usepackage[utf8]{inputenc}
\usepackage[english]{babel}
\usepackage{amssymb}
\usepackage{amsfonts}
\usepackage{tikz-cd}
\usepackage{enumerate}
\usepackage[utf8]{inputenc}
\usepackage{tikz}
\usepackage[top=1 in,bottom=1in, left=1 in, right=1 in]{geometry}

\usepackage{fancyhdr}
\usepackage{graphicx}
\usepackage{fancyhdr}

\usepackage{dynkin-diagrams}
\usetikzlibrary{decorations.markings}

\usepackage{hyperref}

\renewcommand{\rm}[1]{\mathrm{#1}}

\newcommand{\R}{\mathbb{R}}
\newcommand{\p}{\mathbb{P}}
\newcommand{\N}{\mathbb{N}}

\newcommand{\B}{\mathcal{B}}
\newcommand{\C}{\mathbb{C}}
\newcommand{\s}{\mathcal{O}}

\newcommand{\Z}{\mathbb{Z}}

\theoremstyle{plain}
\newtheorem{theorem}{Theorem}[section] 

\theoremstyle{definition}
\newtheorem{prop}[theorem]{Proposition}
\newtheorem{lemma}[theorem]{Lemma}
\newtheorem{conj}[theorem]{Conjecture}

\newtheorem{Cor}[theorem]{Corollary}

\newtheorem{Remark}[theorem]{Remark}

\newtheorem{cor}[theorem]{Corollary}

\newtheorem{defn}[theorem]{Definition} 
\newtheorem{exmp}[theorem]{Example}

\title{On Representations of Weyl Groups attached to Spherical Varieties}

\author{Guy Kapon, Guy Shtotland}
\date{}

\begin{document}

\maketitle

\begin{abstract}
    To any smooth spherical $G$ variety $X$ we attach two representations of the Weyl group of $G$. The first is constructed geometrically using Borel Moore homology. The second is constructed combinatorially using the structure of the Borel orbits on $X$. Using the Fourier Sato transform, we show that both representations are isomorphic. We use this result to translate (in the cotangent case) a conjecture in the relative Langlands program proposed by \cite{FGT} to a combinatorial problem. We solve this combinatorial problem for several examples. 
\end{abstract}

\begin{figure}[h]
    \centering
    \includegraphics[width=0.5\textwidth]{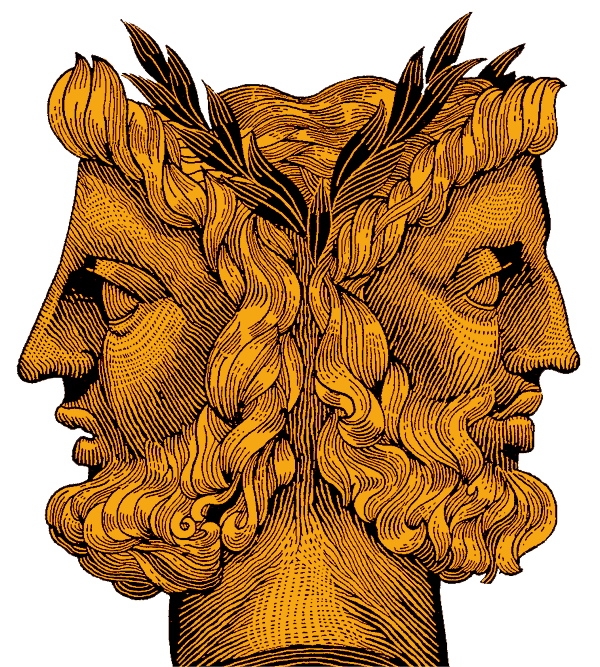}
    \caption{Janus. Artist and original publication unknown; modern colorized reproduction of an engraving.}
    \label{fig:myimage}
\end{figure}

\section{Introduction}

In the influential work \cite{benzvi2024relativelanglandsduality} the authors introduced relative Langlands duality, attaching to a reductive group $G$ acting on a spherical variety $X$ a dual Hamiltonian space $M^\vee$ with a $G^\vee$ action. This work gave rise to many conjectures, one of them suggested in \cite{FGT}. They conjecture that two representations of the Weyl group of $G$, the first constructed from the geometry of the cotangent bundle $T^*X$, and the second from the geometry of $M^\vee$ are isomorphic (up to a twist by sign). In this work we provide a method to translate this geometric problem into a combinatorial one and solve it in some examples. Recently, in \cite{braverman2026relativelanglandsdualitykoszul} they proved the same conjecture under various assumptions, including the assumption that the relative derived Satake conjecture (Conjecture 1.1.2 in \cite{braverman2026relativelanglandsdualitykoszul}) holds for $X$. Our approach is not conditional on relative derived Satake.

\subsubsection*{Main Results}

Let $G$ be a connected reductive group defined over $\C$.

Let $\mathfrak{g}$ be the Lie algebra of $G$ and let $\mathfrak{g}^*$ be its dual.

Denote by $W$ the Weyl group of $G$.

Let $\B$ be the flag variety of $G$, and choose a Borel subgroup $B\in \B$.  Let $\Tilde{N}=T^*\B$ be the cotangent bundle of $\B$. We consider the moment map $\mu:\Tilde{N}\rightarrow \mathfrak{g}^*$ and the Steinberg variety $St=\Tilde{N}\times_{\mathfrak{g}^*}\Tilde{N}$.

There is an algebra structure on the top Borel Moore homology, $H_{top}^{BM}(St)$, given by convolution, it is well known that $H_{top}^{BM}(St)\cong \C[W]$ (see \cite{Chriss1997RepresentationTA}).

Let $X$ be a smooth spherical $G$ variety, i.e. a smooth variety with finitely many Borel orbits. The cotangent bundle $T^*X$ also has a moment map $T^*X\rightarrow \mathfrak{g}^*$. Consider $\Lambda=T^*X\times_{\mathfrak{g}^*}\Tilde{N}$.

\begin{defn}
    Consider $T^*X\times_{\mathfrak{g}^*}\Tilde{N}\times_{\mathfrak{g}^*}\Tilde{N}$ with its three natural projections $\pi_0,\pi_1,\pi_2$ from it. The projection $\pi_0$ goes to $St$ and the projections $\pi_1,\pi_2$ to $\Lambda$. We have a module structure of $H_{top}^{BM}(\Lambda)$ over $H_{top}^{BM}(St)$ given by $x*y=\pi_{1*}(\pi_0^*x\cap \pi_2^*y)\in H_{top}^{BM}(\Lambda)$ for $x\in H_{top}^{BM}(St), y\in H_{top}^{BM}(\Lambda)$. 

    This gives a module over $\C[W]$, denote it by $V_{geo}(X)$.
\end{defn}

There is another way to construct a representation of $W$ from $X$, using the structure of the Borel orbits on $X$. We prove that these two representations are isomorphic. Let us describe this second construction.

In \cite{knop} Knop defined an action of $W$ on $B\backslash X$, the set of Borel orbits on $X$. For $w\in W$ and $\s\in B\backslash X$ we denote this action by $w\s$. We denote by $\overline{\s}$ the closure of $\s$ in the Zariski topology.

Let $\Delta$ be the set of simple roots of $G$, for any $s\in\Delta$ and every $\s\in B\backslash X$ we can associate a type to $\s$ with respect to $s$. 

Let $g\in G$ such that $\s=BgH$. Let $B\subset P_s$ be the parabolic corresponding to $s$. We have the following (see \cite{knop} for a proof). 

\begin{defn}\label{type_rank_1_spheric}

Consider the action of $P_s$ on $\s$, there are several options.

\begin{enumerate}
    \item (Type G) $P_s\s=\s$.
    \item (Type U1) $P_s\s=\s\cup s\s$ and $\s\subset \overline{s\s}$.
    \item (Type U2) $P_s\s=\s\cup s\s$ and $s\s\subset \overline{\s}$.
    \item (Type T1) $P_s\s=\s\cup s\s\cup \s'$ and $s\s\cup \s \subset \overline{\s'}$.
    \item (Type T2) $P_s\s=\s\cup \s'\cup s\s'$ and $s\s'\cup \s' \subset \overline{\s}$.
    \item (Type N1) $s\s=\s$, $P_s\s=\s\cup \s'$ and $ \s \subset \overline{\s'}$.
    \item (Type N2) $s\s=\s$, $P_s\s=\s\cup \s'$ and $\s' \subset \overline{\s}$.
\end{enumerate}
    
\end{defn}

\begin{defn}
    We define an action of $\C[W]$ on the vector space generated by the orbits $\C[B\backslash X]$. For $s\in \Delta$ we define $(s+1)[\s]=\alpha_{s,\s}(\sum_{\s'\in P_s\s}[\s'])$. The value $\alpha_{s,\s}$ is defined according to the type of $\s$ with respect to $s$ as follows: 

    \begin{enumerate}
    \item In Types G and N1, $\alpha_{s,\s}=2$.
    \item In Types U1,U2, and T1, $\alpha_{s,\s}=1$.
    \item In Types T2 and N2, $\alpha_{s,\s}=0$.
\end{enumerate}

This gives a well defined action of $\C[W]$ on $\C[B\backslash X]$ (as follows from results of \cite{knop}). Denote this module by $V_{comb}(X)$. 
\end{defn}

Our main result is the following.

\begin{theorem}\label{main}
    We have an isomorphism of $\C[W]$ modules $V_{comb}(X)\cong V_{geo}(X)$.
\end{theorem}

\begin{Remark}
    It is well known that both representations have the same dimension which is equal to the number of Borel orbits on $X$. For $V_{comb}(X)$ this is clear. For $V_{geo}(X)$, Proposition \ref{basis of top BM} gives a basis for $V_{geo}(X)$, namely the fundamental classes of the irreducible components of $\Lambda$. Let $m:T^*X\times\Tilde{N}\rightarrow\mathfrak{g}^*$ be the moment map; it is the sum of the moment map of $\Tilde{N}$ and of $T^*X$. Clearly $\Lambda\cong m^{-1}(0)$. The irreducible components of $m^{-1}(0)$ are given by the closures of the conormal bundles $T^*_\s(X\times\B)$ for $\s\subset X\times \B$ a $G$ orbit.
\end{Remark}

Next, we describe a conjecture by \cite{FGT}.

 Let $G^\vee$ be the Langlands dual group of $G$. Under certain assumptions on $X$, \cite{benzvi2024relativelanglandsduality} introduce a dual Hamiltonian variety $M^\vee$ with a $G^\vee$ action and moment map $\mu^\vee:M^\vee\rightarrow \mathfrak{g}^{\vee*}$. Using $M^\vee$ we define another representation of $W$. Consider $\Lambda^\vee=M^\vee\times_{\mathfrak{g}^{\vee*}}\Tilde{N^\vee}$. Then $H^{BM}_{top}(\Lambda^\vee)$ is a module over $H^{BM}_{top}(St^\vee)\cong \C[W]$. Here, $St^\vee$ and $\Tilde{N^\vee}$ are defined as $St$ and $\Tilde{N}$ but for the group $G^\vee$ instead of $G$.

\begin{conj}[\cite{FGT}]\label{main conj}
    There is an isomorphism of $W$ representations 
    $H^{BM}_{top}(\Lambda^\vee)\cong sgn\otimes H^{BM}_{top}(\Lambda)$.
\end{conj}

We are interested in a special case of this conjecture in which $M^\vee$ is of the form $T^*X^\vee$ for some spherical $G^\vee$ variety $X^\vee$. In this case, we can verify Conjecture \ref{main conj} by relating the geometric constructions to the combinatorial ones and then solving a combinatorial problem. We use this idea to prove:

\begin{theorem}\label{application}
    Conjecture \ref{main conj} holds for the following dual pairs of spherical varieties:
    \begin{enumerate}
        \item $Hom(\C^n,\C^n)$ as a $GL_n\times GL_n$ variety is dual to $GL_n\times \C^n$ as a $(GL_n\times GL_n)^\vee\cong GL_n\times GL_n$ variety.
        \item $GL_{2n+1}/GL_{n+1}\times GL_n$ as a $GL_{2n+1}$ variety is dual to $GL_{2n+1}/Sp_{2n}$ as a $GL_{2n+1}^\vee\cong GL_{2n+1}$ variety.
    \end{enumerate}
\end{theorem}

\begin{Remark}
    The first item of Theorem \ref{application} is covered by the results of \cite{braverman2026relativelanglandsdualitykoszul} as relative derived Satake is known in this case (see \cite{Braverman_Finkelberg_Ginzburg_Travkin_2021}).
\end{Remark}

Based on the calculations used to prove Theorem \ref{application}, we propose a conjecture for the cotangent case, relating the Borel orbits on $X$ and on $X^\vee$. In order to formulate our conjecture we recall the notion of the rank of a $B$ orbit. See Section 2 of \cite{knop} for more detail.

\begin{defn}\label{Borel rank}
    Let $Y$ be a variety with an action of $B$. Let $k(Y)$ be the ring of rational functions on $Y$. Let $\chi(Y)$ be the lattice of characters $\chi$ of $B$ such that there is $0\neq f\in k(Y)$ with $bf=\chi(b)f$ for every $b\in B$. Define $\operatorname{rank}(Y)=\dim(\chi(Y))$.
\end{defn}

\begin{conj}
    Let $X$ be a smooth spherical $G$ variety such that $T^*X$ is hyperspherical (see \cite{benzvi2024relativelanglandsduality} for a definition) with a relative Langlands dual of the form $T^*X^\vee$ for some smooth spherical $G^\vee$ variety $X^\vee$. Denote by $B\backslash X$ and $B^\vee\backslash X^\vee$ the sets of Borel orbits on $X$ and $X^\vee$ respectively. There exists a bijection $\Phi:B\backslash X\rightarrow B^\vee\backslash X^\vee$ which intertwines Knop's action of the Weyl group such that:
    \begin{enumerate}
        \item $\operatorname{rank}(x)+\operatorname{rank}(\Phi(x))=\operatorname{rank}(G)$.
        \item Let $s\in W$ be a simple reflection, $x\in B\backslash X$ has type $G$ with respect to $s$ if and only if $\Phi(x)$ has type $T2$ with respect to $s$. The opposite also holds.
    \end{enumerate}
\end{conj}

We also mention the paper \cite{shtotland2026relativekazhdanlusztigisomorphism} by the second author. In that paper the following conjecture appears.

\begin{conj}\label{second conj}
    We have an isomorphism of $W$ representations $V_{comb}(X)\cong sgn\otimes H^{BM}_{top}(\Lambda^\vee)$
\end{conj}

It is clear that Theorem \ref{main} implies that Conjectures \ref{main} and \ref{second conj} are equivalent. 

In \cite{shtotland2026relativekazhdanlusztigisomorphism} there is also a conjecture about the representations of the affine Weyl group obtained from $K^{G}(\Lambda)$. We plan to apply the methods of the current paper to try and establish the conjecture for $K$ theory for some cases in a future work.

On the topic of related works, we mention the independent work 
\cite{ma2025thetacorrespondencespringercorrespondence}. Theorem \ref{main} is not stated in this paper but under some assumptions on $X$ it can be deduced from its results, see Proposition 3.25. They assume that the space $X$ has no roots of type $N$, we do not need this assumption and we prove Theorem \ref{main} for any smooth spherical variety.

\subsubsection*{Idea of the Proof}
    We outline the main ideas of the proof of Theorem \ref{main}. Our first step is a reduction to the case of $X$ being a homogeneous space, this is done by introducing filtrations on $V_{geo}(X)$ and $V_{comb}(X)$ indexed by the $G$ orbits on $X$.

     Assume that $X=G/H$. Both $V_{geo}(X)$ and $V_{comb}(X)$ have a natural basis parametrized by $B$ orbits on $X$. However, the action written in these bases is not the same. In fact, looking at examples, we see that an isomorphism between the representations given in these bases cannot be sparse, i.e. the image of a basis element has to be a sum of many basis elements. 

    To address this, we use a Fourier transform, after which the isomorphism comes from an identification of natural bases of both representations.

     Let $\mathfrak{h}$ be the Lie algebra of $H$. Let $\Tilde{\mathfrak{g}}=\{(x,B)|x\in \mathfrak{g},B\in \B, x\in Lie(B)\}$ be the Grothendieck-Springer resolution.

    We use the Fourier Sato transform to show that $V_{geo}(X)$ is isomorphic to the representation given by the action of $H^{BM}_{top}(\Tilde{\mathfrak{g}}\times_{\mathfrak{g}}\Tilde{\mathfrak{g}})\cong \C[W]$ on $H^{BM}_{top}(\mathfrak{h}\times_{\mathfrak{g}}\Tilde{\mathfrak{g}})$. 
    In general, the computation of the convolution action in this description can be more complicated geometrically, as it might involve computing non-transversal intersections. However, using a direct computation for some cases, and combinatorial arguments, we can deduce the whole action of any simple reflection and conclude that $H^{BM}_{top}(\mathfrak{h}\times_{\mathfrak{g}}\Tilde{\mathfrak{g}})\cong V_{comb}(X)$.

\subsubsection*{Structure of the Paper}
    \begin{itemize}
        \item In Section \ref{s2} we reduce to the homogeneous case and use the Fourier Sato transform to relate $V_{geo}(X)$ to $H^{BM}_{top}(\mathfrak{h}\times_{\mathfrak{g}}\Tilde{\mathfrak{g}})$. 

         \item  In Section \ref{s3} we relate $H^{BM}_{top}(\mathfrak{h}\times_{\mathfrak{g}}\Tilde{\mathfrak{g}})$ to $V_{comb}(X)$ and prove Theorem \ref{main}.

        \item  In Section \ref{s4} we compute some examples and prove Theorem \ref{application}.

        \item In Appendix \ref{A1} we recall properties of Borel Moore homology that are used throughout the paper.

    \end{itemize}

\subsubsection*{Acknowledgments}

We would like to thank Shachar Carmeli, Eitan Sayag, and Noam Zimhoni for helpful discussion.

G.S. was partially supported by ISF grant number 1781/23 during the work on this paper.

G.K. was partially supported by ISF Beresheet grant 4093/25 and BSF grant 2024766.

\section{The homogeneous case and the Fourier Sato transform}\label{s2}

In this section, we use the Fourier Sato transform to give another description of $V_{geo}(X)$ that is used to compare it with $V_{comb}(X)$. Before doing that, we need to reduce to the homogeneous case.

We can write $X$ as a union of $G$ orbits $X=\cup_{i=1}^l X_i$. It is clear that $V_{comb}(X)=\oplus_i V_{comb}(X_i)$; it is enough to show the following.

\begin{prop}
    We have an isomorphism of $W$ representations 
    $V_{geo}(X)\cong\oplus_i V_{geo}(X_i)$.
\end{prop}

\begin{proof} 
    Let $X_1$ be a closed $G$ orbit in $X$. We have $\Lambda=T^*X\times_{\mathfrak{g}^*}\tilde{N}$, let $p:\Lambda \to X$ be the natural projection. Inside $V_{geo}(X)$ we have the subspace $H^{BM}_{top}(p^{-1}(X_{1}))$. It is enough to show that this subspace is a subrepresentation, isomorphic to $V_{geo}(X_{1})$, and that the quotient is isomorphic to $V_{geo}(X\setminus X_{1})$. After we show this, we can finish the proof using induction and the fact that all representations of $W$ are semisimple.

     Since $X_{1}$ is closed, so is $p^{-1}(X_{1})$; the action of $W$ preserves $H^{BM}_{top}(p^{-1}(X_{1}))$ meaning it is a subrepresentation. 

    The fact that the quotient by $H^{BM}_{top}(p^{-1}(X_{1}))$ is isomorphic to $V_{geo}(X\setminus X_{1})$ is also clear.

    It remains to show that $H^{BM}_{top}(p^{-1}(X_{1}))\cong V_{geo}(X_{1})$.

    Consider the projection $q:p^{-1}(X_{1})\to T^*X_{1}\times_{\mathfrak{g}^*} \tilde{N}$. Taking the pullback, we get a map $V_{geo}(X_{1}) \to H^{BM}_{top}(p^{-1}(X_{1}))$. Both sides have bases corresponding to $B$ orbits on $X_{1}$ and this is an isomorphism of vector spaces. It is easy to check that this map is indeed a map of $H^{BM}_{top}(St)$ modules.

\end{proof}

\begin{Remark}
    The more natural result is that $V_{geo}(X)$ has a filtration and its graded parts are $V_{geo}(X_{i})$. We are working in a situation where all representations are semisimple, so a filtration always splits. In other settings this might not be the case, and one may only have a filtration. For example, if one considers equivariant $K$ theory instead of Borel Moore homology as a representation of the affine Weyl group.
\end{Remark}
From now on we assume that $X=G/H$ is a homogeneous space.

Let $\mathfrak{h}$ be the Lie algebra of $H$. We have $\Lambda=G\times^H(\mathfrak{h}^\perp\times_{\mathfrak{g}^*}\Tilde{N})$. In particular, $H_{top}^{BM}(\Lambda)=H_{top}^{BM}(\mathfrak{h}^\perp\times_{\mathfrak{g}^*}\Tilde{N})$. Let $j$ be the embedding $j:\mathfrak{h}^\perp\rightarrow \mathfrak{g}^*$. 

Let $C_{\Tilde{N}}=\C[dim(\Tilde{N})],C_{\mathfrak{h}^\perp}=\C[dim(\mathfrak{h}^\perp)]$ be the constant sheaves on $\Tilde{N}$ and $\mathfrak{h}^\perp$ respectively, shifted such that the resulting sheaf is perverse.

\begin{prop}\label{pass to Ext}

 We have an algebra isomorphism $H^{BM}_{top}(St)\cong \text{Ext}^{0}(\mu_*C_{\Tilde{N}}, \mu_*C_{\Tilde{N}})$. Moreover, the following diagram is commutative and the vertical lines are isomorphisms. 

    \[
    \begin{tikzcd}
H^{BM}_{top}(St) \otimes H^{BM}_{top}(\mathfrak{h}^\perp\times_{\mathfrak{g}^*}\Tilde{N}) 
    \arrow[r, "\textit{convolution action}"] 
    \arrow[d]
& H^{BM}_{top}(\mathfrak{h}^\perp\times_{\mathfrak{g}^*}\Tilde{N}) 
    \arrow[d] \\
\text{Ext}^{0}(\mu_*C_{\Tilde{N}}, \mu_*C_{\Tilde{N}}) \otimes \text{Ext}^{dim(G)-dim(H)}(\mu_*C_{\Tilde{N}}, j_*C_{\mathfrak{h}^\perp}) 
    \arrow[r, "\textit{composition}"] 
& \text{Ext}^{dim(G)-dim(H)}(\mu_*C_{\Tilde{N}}, j_*C_{\mathfrak{h}^\perp})
\end{tikzcd}
\]

The Ext groups are computed in the category ${D_c^b(\mathfrak{g}^{*})}$ of bounded derived constructible sheaves on $\mathfrak{g}^{*}$.

\end{prop}

\begin{proof}
    The map $\mu$ is proper and the map $j$ is a closed embedding, in particular it is proper. 

    Following the proof of Lemma 8.6.1 of \cite{Chriss1997RepresentationTA} we see that, in general for $M_1,M_2$ smooth and proper over a smooth variety $Z$, $C_{M_1}=\C[dim(M_1)],C_{M_2}=\C[dim(M_2)]$ we have:
    
    $$H_{top}^{BM}(M_1\times_Z M_2)=\text{Ext}^{dim(M_1)+dim(M_2)-2dim(M_1\times_Z M_2)}(C_{M_1},C_{M_2})$$

    It is easy to see that $dim(\Tilde{N})=2dim(G)-2dim(B),dim(\mathfrak{h}^\perp)=dim(G)-dim(H)$ and $dim(\mathfrak{h}^\perp\times_{\mathfrak{g}^*}\Tilde{N})=dim(G)-dim(B)$.

    Thus, the result follows using Proposition 8.6.35 of \cite{Chriss1997RepresentationTA}. 
\end{proof}

Next we introduce the Fourier Sato transform. For more details about it, see Subsection 2.7 of \cite{Achar_2013} and Subsection 3.7 of \cite{KashiwaraSchapira1990}.

\begin{defn}
    Let $V$ be a finite dimensional vector space over $\C$. Let $D^b_{con}(V)$ be the derived category of constructible sheaves with respect to the stratification given by the orbits of the scalar multiplication action of $\C^\times$ on $V$. These are called conical (derived) sheaves. 

    Let $V^*$ be the dual vector space. Consider $Q=\{(v,v^*)|Re(v^*(v))\leq 0\}\subset V\times V^*$. Let $q:Q\rightarrow V,q':Q\rightarrow V^*$ be the natural projections. Define the Fourier Sato transform to be $F:D^b_{con}(V)\rightarrow D^b_{con}(V^*)$ given by $F=q'_!q^*[dim(V)]$. 
\end{defn}

We have the following (see Theorem 3.7.9 of \cite{KashiwaraSchapira1990}).

\begin{theorem}
    The Fourier Sato transform $F:D^b_{con}(V)\rightarrow D^b_{con}(V^*)$ is an equivalence of categories. It has an inverse $F^\vee:D^b_{con}(V^*)\rightarrow D^b_{con}(V)$, given by $F^\vee=a_*q_!q'^*[dim(V)]$, here $a:D^b_{con}(V)\rightarrow D^b_{con}(V)$ is induced by multiplication by $-1$.
\end{theorem}

Notice that the sheaves that interest us, i.e. $j_*\omega_1$ and $\mu_*\omega_2$ are conical on $\mathfrak{g}^*$. Thus, we can apply the Fourier Sato transform to them and get sheaves on $\mathfrak{g}$.

Over $\mathfrak{g}$ we have $\Tilde{\mathfrak{g}}=\{(g,B)|g\in \mathfrak{g}, B\in \B, g\in Lie(B)\}$ with a projection map $\pi:\Tilde{\mathfrak{g}}\rightarrow \mathfrak{g}$. We also have the embedding $i:\mathfrak{h}\rightarrow\mathfrak{g}$. 

Let $C_{\Tilde{\mathfrak{g}}}=\C[dim(\Tilde{\mathfrak{g}})]$ and $C_\mathfrak{h}=\C[dim(\mathfrak{h})]$ be the constant sheaves on $\Tilde{\mathfrak{g}}$ and $\mathfrak{h}$ respectively, shifted such that the resulting sheaf is perverse. 

\begin{prop}\label{trasnformed sheaves}
    \begin{enumerate}
        \item $F^\vee(\mu_*C_{\Tilde{N}})=\pi_*C_{\Tilde{\mathfrak{g}}}$.
        \item $F^\vee(j_*C_{\mathfrak{h}^\perp})=i_*C_{\mathfrak{h}}$

        \item There is an isomorphism $\text{Ext}^{dim(G)-dim(H)}(j_*C_{\mathfrak{h}^\perp},\mu_*C_{\Tilde{N}})\cong \text{Ext}^{dim(G)-dim(H)}(i_*C_{\mathfrak{h}},\pi_*C_{\Tilde{\mathfrak{g}}})$ of $\text{Ext}^{0}(\mu_*C_{\Tilde{N}},\mu_*C_{\Tilde{N}})\cong \text{Ext}^{0}(\pi_*C_{\Tilde{\mathfrak{g}}},\pi_*C_{\Tilde{\mathfrak{g}}})$ modules.
    \end{enumerate}
\end{prop}

\begin{proof}
    The first part follows from Lemma 2.2 of \cite{Achar_2013}. Clearly the third part follows from the previous two. We only prove the second part.

    We show that applying the Fourier Sato transform on $\mathfrak{g}$ we have $F(i_*\C[dim(\mathfrak{h}])=j_*\C[dim(\mathfrak{h}^\perp)]$. 

    Consider $i:\mathfrak{h}\rightarrow\mathfrak{g}$ and the adjoint map $i':\mathfrak{g}^*\rightarrow\mathfrak{h}^*$. By 2.14 of \cite{Achar_2013} we have $F(i_*\C[-dim(\mathfrak{g})])=i'^*F_{\mathfrak{h}}(\C[-dim(\mathfrak{h})])$, here $F_{\mathfrak{h}}$ is the Fourier transform between $\mathfrak{h}$ and $\mathfrak{h}^*$. Up to a shift, the Fourier transform of the constant sheaf is the skyscraper sheaf at zero. More precisely, let $z:0\rightarrow \mathfrak{h}^*$ be the embedding of zero, then $F_{\mathfrak{h}}(\C[-dim(\mathfrak{h})])=z_*\C[-2dim(\mathfrak{h})]$. 

    We have the following commutative diagram. 

        \[
    \begin{tikzcd}
\mathfrak{h}^\perp 
    \arrow[r,] 
    \arrow[d,"j"]
& 0  
    \arrow[d,"z"] \\
\mathfrak{g}^* 
    \arrow[r,"i'"] 
& \mathfrak{h}^*
\end{tikzcd}
\]

By base change we have $i'^*z_*\C[-2dim(\mathfrak{h})]=j_*\C[-2dim(\mathfrak{h})]$. Thus $F(i_*\C)=j_*\C[dim(\mathfrak{g})-2dim(\mathfrak{h})]=j_*\C[dim(\mathfrak{h}^\perp)-dim(\mathfrak{h})]$
\end{proof}

Denote $St'=\Tilde{\mathfrak{g}}\times_{\mathfrak{g}} \Tilde{\mathfrak{g}}$ and $\Lambda'=\mathfrak{h}\times_{\mathfrak{g}}\Tilde{\mathfrak{g}}$. We refer to $St'$ as the Grothendieck-Steinberg variety. 

Similarly to Proposition \ref{pass to Ext} we obtain the following.

\begin{prop}\label{pass to Ext again}

We have an algebra isomorphism $H^{BM}_{top}(St')\cong \text{Ext}^{0}(\pi_*\C, \pi_*\C)$. Moreover, the following diagram is commutative and the vertical lines are isomorphisms.
    \[
    \begin{tikzcd}
H^{BM}_{top}(St') \otimes H^{BM}_{top}(\mathfrak{h}\times_{\mathfrak{g}}\Tilde{\mathfrak{g}}) 
    \arrow[r, "\textit{convolution action}"] 
    \arrow[d]
& H^{BM}_{top}(\mathfrak{h}\times_{\mathfrak{g}}\Tilde{\mathfrak{g}}) 
    \arrow[d] \\
\text{Ext}^{0}(\pi_*\C, \pi_*\C) \otimes \text{Ext}^{dim(G)-dim(H)}(i_*C_{\mathfrak{h}},\pi_*C_{\Tilde{\mathfrak{g}}})
    \arrow[r, "\textit{composition}"] 
& \text{Ext}^{dim(G)-dim(H)}(i_*C_{\mathfrak{h}},\pi_*C_{\Tilde{\mathfrak{g}}})
\end{tikzcd}
\]

The Ext groups are computed in the category ${D_c^b(\mathfrak{g})}$ of bounded derived constructible sheaves on $\mathfrak{g}$.

\end{prop}

\begin{proof}
    The proof is almost identical to the proof of Proposition \ref{pass to Ext}, the only difference is the dimension computation. Clearly $dim(\Tilde{\mathfrak{g}})=dim(G)$, $dim(\mathfrak{h})=dim(H)$ and by Corollary \ref{cor: irreducible components} we have $dim(\mathfrak{h}\times_{\mathfrak{g}}\Tilde{\mathfrak{g}})=dim(H)$.
\end{proof}

As a corollary we obtain the following.

\begin{prop}\label{Fourier isomorphism}
    \begin{itemize}
        \item We have an algebra isomorphism $H^{BM}_{top}(St')\cong H^{BM}_{top}(St)$.
        \item We have an isomorphism of modules $H^{BM}_{top}(\mathfrak{h}\times_{\mathfrak{g}}\Tilde{\mathfrak{g}})\cong H^{BM}_{top}(\mathfrak{h}^\perp\times_{\mathfrak{g}^*}\Tilde{N})$.
    \end{itemize}
\end{prop}

\begin{proof}
    This follows from Propositions \ref{pass to Ext}, \ref{trasnformed sheaves} and \ref{pass to Ext again}.
\end{proof}

\section{Explicit Computation}\label{s3}
In this section, we explicitly compute the $W$ representation $H^{BM}_{top}(\mathfrak{h}\times_{\mathfrak{g}}\Tilde{\mathfrak{g}})$ in the basis given by the irreducible components of $\mathfrak{h}\times_{\mathfrak{g}}\Tilde{\mathfrak{g}}$.  From this computation we deduce Theorem \ref{main}.
\subsection{The Steinberg variety}
In this subsection, we describe the irreducible components of the Grothendieck-Steinberg variety, and the ring structure on its top Borel Moore homology.

By Proposition \ref{basis of top BM} the fundamental classes of the irreducible components of $St'$ give a basis for $H_{top}^{BM}(St')$.

 Let $T\subset B$ be a maximal torus of our chosen Borel $B$. We define $W=N_G(T)/T$, this is the Weyl group of $G$.

\begin{defn}
Let $w\in W$, we say that two $B_1,B_2\in \B$ are in relative position $w$ if there is $g\in G$ such that $B_1=g^{-1}Bg$ and $B_2=g^{-1}w^{-1}Bwg$. We denote this by $B_1\sim^wB_2$. 
\end{defn}

The following is well known and easy to prove.
\begin{prop}
    Let $w\in W$, $B_1,B_2\in \B$ and $g\in G$, if $B_1\sim^wB_2$ then $g^{-1}B_1g\sim^w g^{-1}B_2g$.
\end{prop}

Denote by $\mathfrak{g}^{rs}$ the dense subset of regular semisimple elements.

\begin{prop}

The irreducible components of $St'$ are given by $$St'_w=\overline{\{(t,B_1,B_2)|t\in \mathfrak{g}^{rs}, t\in Lie(B_1)\cap Lie(B_2), B_1\sim^wB_2\}}$$ 

The product in $H_{top}^{BM}(St')$ is given by $[St'_{w_1}][St'_{w_2}]=[St'_{w_1w_2}]$.

In particular,  $H_{top}^{BM}(St)\cong H_{top}^{BM}(St')\cong \C[W]$.

\end{prop}

\begin{proof}
    It is clear that the irreducible components of the restriction of $St'$ to the regular semisimple locus are the restrictions of $St'_w$. As the regular semisimple elements are dense, the first part follows.

    The more interesting statement is $[St'_{w_1}][St'_{w_2}]=[St'_{w_1w_2}]$. Again it is enough to check this over the regular semisimple elements. Let $t$ be regular semisimple and let $B_1,B_2,B_3$ be three Borel subgroups whose Lie algebras contain $t$ such that $B_1\sim^{w_1} B_2$ and $B_2\sim^{w_2} B_3$. We claim that $B_1\sim^{w_1w_2}B_3$. 

    Let $T'$ be a maximal torus such that $t\in Lie(T')$. The element $t$ is regular semi-simple so we must have $T'\subset B_1\cap B_2\cap B_3$. We can choose $g\in G$ such that $g^{-1}Tg=T'$. Conjugation by $g$ gives a bijection between Borel subgroups containing $T$ and those containing $T'$. Consider $gB_1g^{-1}$, $gB_2g^{-1}$, $gB_3g^{-1}$, these are Borels containing $T$. Thus, there are $x_1,x_2,x_3\in N_G(T)$ such that $gB_1g^{-1}=x^{-1}_1Bx_1$, $gB_2g^{-1}=x^{-1}_2Bx_2$, and  $gB_3g^{-1}=x^{-1}_3Bx_3$. Changing $g$ to $x_1g$ we can assume that $x_1=1$. The relative position of any two Borel subgroups is preserved under conjugation, so we must have $x_2=w_1$ and $x_3=w_1w_2$ in the Weyl group $W$. The result follows. 
\end{proof}

It is useful to write down the irreducible component explicitly in the case of a simple reflection.

\begin{defn}\label{def:functionl alpha_s}
For any root $\alpha$ and $(t,B)\in \Tilde{\mathfrak{g}}$ we define $\alpha_B(t)$ by identifying the Lie algebra $B/[B,B]$ with the canonical Cartan of $G$ and using the map $Lie(B)\rightarrow Lie(B/[B,B])$.    
\end{defn}

\begin{lemma}\label{simple reflection steinberg}
    Let $s\in W$ be a simple reflection corresponding to a simple root $\alpha$, then: $$St'_{s}=\{(t,B_{1},B_{2})|t\in Lie(B_{1})\cap Lie(B_{2}),B_{1}\sim^{s} B_{2}\}\cup \{(t,B_1,B_1)| t\in Lie(B_1),\alpha_{B_1}(t)=0\}$$
\end{lemma}

\begin{proof}
    We need to compute the closure:

    $$St'_s=\overline{\{(t,B_1,B_2)|t\in \mathfrak{g}^{rs}, t\in Lie(B_1)\cap Lie(B_2), B_1\sim^sB_2\}}$$

    Denote $X_1=\{(t,B_{1},B_{2})|t\in Lie(B_{1})\cap Lie(B_{2}),B_{1}\sim^{s} B_{2}\}$ and $X_2=\{(t,B_1,B_1)| t\in Lie(B_1),\alpha_{B_1}(t)=0\}$. 
    
    It is clear that $X_1\subset St'_s$. We show $X_2\subset St'_s$, notice that for $t,B_1$ with $t\in Lie(B_1)$ and $\alpha_{B_1}(t)=0$ the conjugation of $t$ by $s$ preserves it, and so $t$ lies in the Lie algebra of any Borel subgroup $B_2$ with $B_1\sim^sB_2$.

    Now we show the opposite direction, let $(t,B_1,B_2)\in St'_s$. The closure of the $G$ orbit on $\B\times \B$ of pairs of Borels in relative position $s$ is equal to the union of two $G$ orbits, the diagonal and the orbit of Borels in relative position $s$. Thus, we either have $B_1\sim^s B_2$ or $B_1=B_2$. In the former case we get that $(t,B_1,B_2)\in X_1$. 

    Consider:
    $$ Y=\{(t,B_1,B_2)| t\in Lie(B_1)\cap Lie(B_2), B_1\sim^sB_2\} \subseteq St'_{s}$$

    $Y$ is a dense subset of $St'_s$, and over any fixed $(B_{1},B_{2})$, it is a subvector space of $\mathfrak{g}$. Therefore, the fibers of its closure $St'_{s}$ are also vector spaces over any $(B_{1},B_{2})$. Furthermore, using the action of $G$, we see that these subspaces have constant dimension over the orbit of $\B\times\B$. Finally, $X_{2}\subseteq St'_{s}$ and $X_{2}$ is a codimension one vector space in $Lie(B_1)$ over any $(B_{1},B_{1})\in\B\times\B$ in the diagonal. Therefore, $St'_{s}\cap St'_{1}=X_{2}$ or all of $St'_{1}$, and the second case is clearly not possible. 
    
\end{proof}

\subsection{The variety $\mathfrak{h} \times_{\mathfrak{g}}\Tilde{\mathfrak{g}}$}

In this subsection, we describe some of the geometry of the variety $\Lambda'=\mathfrak{h} \times_{\mathfrak{g}}\Tilde{\mathfrak{g}}$. We also describe several sheaves on $\B$ closely related to $\Lambda'$.

We work with vector bundles over generalized flag varieties and subsheaves of them, all the subsheaves that appear are $H$ equivariant, and in particular are vector bundles when restricted to an $H$ orbit. We use the following notations:

\begin{itemize}
    \item $\mathfrak{g}$ is considered as the constant bundle with fiber $\mathfrak{g}$.
    \item $\mathfrak{b}$ is the bundle whose fiber at a point $B\in \B$ is the Lie algebra $\mathfrak{b}=Lie(B)$ of $B$. 
    \item $\mathfrak{h}$ is the constant bundle $\mathfrak{h}=T_{e}H$.
    
    \item Let $B_0\subseteq P_0$ be a Borel and parabolic, we have the bundle $\mathfrak{p}$ on $\mathcal{P}=P_0\backslash G$ that gives to a parabolic $P\in \mathcal{P}$ the Lie algebra of $P$. Pulling this back using the map $B_0\backslash G\to P_0\backslash G$ we get a bundle on $\B$ which we also denote by $\mathfrak{p}$. Notice that this bundle only depends on the conjugacy class of the pair $(B_0,P_{0})$. For example, for a simple root $\alpha$ we have the bundle $\mathfrak{p}_\alpha(\mathfrak{b})$ attached to $B\subset P_\alpha$ the parabolic corresponding to $\alpha$.
    
    \item For a sheaf $\mathfrak{f}$ as above and an orbit $\s$ we denote by $\mathfrak{f}|_\s$ the restriction of $\mathfrak{f}$ to $\s$.

    \item For a sheaf $\mathfrak{f}$ as above and a point $B'\in \B$ we denote by $\mathfrak{f}_{B'}$ the fiber of $\mathfrak{f}$ at $B'$.
\end{itemize}

Notice that all these bundles are contained in $\mathfrak{g}$. In particular, we may form sums and intersections of these bundles inside the constant bundle $\mathfrak{g}$. These operations may produce sheaves that are not bundles. 

\begin{prop}\label{dimnesion calculation}
    Let $B\subset P$ be a parabolic subgroup of $G$.
     Let $O$ be the $H$ orbit of $P$ inside $\mathcal{P} =P\backslash G$.

    On $B\backslash PH\subseteq\B$ the sheaves $\mathfrak{h}\cap \mathfrak{p}$ and $\mathfrak{h} + \mathfrak{p}$ are bundles of rank $\dim(H) - \dim(O)$ and $\dim(O)+\dim(P)$ respectively.
    
\end{prop}
\begin{proof}
    First, notice that both sheaves are pulled back from $\mathcal{P}$, and are $H$-equivariant and so they are bundles on $B\backslash PH\subseteq\B$.  It remains to calculate the dimensions. We have the exact sequence:

    \[\begin{tikzcd}
	   0 & \mathfrak{h}\cap \mathfrak{p} & \mathfrak{h} \oplus \mathfrak{p} & \mathfrak{h} + \mathfrak{p} & 0
	   \arrow[from=1-1, to=1-2]
	   \arrow[from=1-2, to=1-3]
	   \arrow[from=1-3, to=1-4]
	   \arrow[from=1-4, to=1-5]
    \end{tikzcd}\]
    It is enough to calculate the dimension of $\mathfrak{h} + \mathfrak{p}$. The space $\mathfrak{h} + \mathfrak{p}$ is the tangent space at the identity of $PH$, and therefore has dimension $dim(O)+\dim(P)$ as claimed.
\end{proof}

\begin{Cor}\label{cor: irreducible components}
     For any $\s$ an $H$ orbit on $\B$, let 

        $$ \mathring{X}_{\mathcal{O}}=\{(h,B)|h\in \mathfrak{h}, B\in \s, h\in Lie(B)\}$$
     
    $$ X_{\mathcal{O}}=\overline{\mathring{X}_\s}$$

    The space $\mathfrak{h}\times_{\mathfrak{g}}\Tilde{\mathfrak{g}}$ is equidimensional, its irreducible components are $X_\s$ and their dimension is $dim(H)$.

\end{Cor}
\begin{proof}
    Since $\mathfrak{h}\cap\mathfrak{b}$ is a bundle when restricted to an $H$ orbit $\mathcal{O}$, it follows from Proposition \ref{dimnesion calculation} that the space $\mathring{X}_\s$ is irreducible and of dimension $\mathrm{dim}(H)$. These sets are a disjoint cover of $\mathfrak{h}\times_{\mathfrak{g}}\Tilde{\mathfrak{g}}$, and are all of the same dimension. Thus, their closures are exactly the irreducible components.
\end{proof}
\begin{Cor}\label{cor: non dense}
    Let $\alpha$ be a simple root, let $\s$ be a $B$ orbit on $G/H$, we can also consider it as a $H$ orbit on $\B$. If $\mathcal{O}$ is not dense in $P_{\alpha}\mathcal{O}$, then $\mathfrak{p}_{\alpha}(\mathfrak{b})\cap \mathfrak{h} =\mathfrak{b} \cap \mathfrak{h}$ on $\s$.
\end{Cor}
\begin{proof}
    The inclusion of $\mathfrak{b} \cap \mathfrak{h}$ in $\mathfrak{p}_{\alpha}(\mathfrak{b})\cap \mathfrak{h}$ is obvious. On the other hand, by Proposition \ref{dimnesion calculation} we have:

    $$ \rm{dim}(\mathfrak{p}_{\alpha}(\mathfrak{b})\cap \mathfrak{h})=\rm{dim}(H)-\rm{dim}(P_s\backslash P_{s}\mathcal{O})$$

    $$ \rm{dim}(\mathfrak{b}\cap \mathfrak{h})=\rm{dim}(H)-\rm{dim}(\mathcal{O})$$

    Since $\mathcal{O}$ is not dense in $P_{s}\mathcal{O}$ and $\mathrm{dim}(P_{s}/B)=1$ we obtain $\rm{dim}(P_s\backslash P_{s}\mathcal{O})=\rm{dim}(\mathcal{O})$. 

    We conclude that the bundles have the same dimension and therefore are equal.

\end{proof}

Let $p:\mathfrak{h}\times_{\mathfrak{g}}\Tilde{\mathfrak{g}}\rightarrow\B$ be the projection on the second coordinate.

\begin{prop}\label{orbit limit}
    Let $\s,\s'$ be two different $H$ orbits on $\B$ such that $\mathcal{O}\subseteq \overline{\mathcal{O}'}$. Then, $X_{\mathcal{O'}}\cap p^{-1}(\mathcal{O})$ is a bundle over $\mathcal{O}$, it is a subbundle of $\mathfrak{b}\cap\mathfrak{h}$, and $X_{\mathcal{O'}}\cap p^{-1}(\mathcal{O})$ is of codimension at least one in $\mathfrak{b}\cap\mathfrak{h}$. We denote it by:
    $$(\mathfrak{b}\cap\mathfrak{h})^{\s'}_\s$$
    
\end{prop}

\begin{proof}
    It is clear that over $\s$, $X_{\mathcal{O'}}\cap p^{-1}(\mathcal{O})$ is a sub-bundle of $\mathfrak{b}\cap \mathfrak{h}$. If the rank of $X_{\mathcal{O'}}\cap p^{-1}(\mathcal{O})$ is $\dim(H)-\dim(\mathcal{O})$ then we have $X_{\mathcal{O}}\subseteq X_{\mathcal{O}'}$ which is impossible, so the rank is at most $\dim(H)-\dim(\mathcal{O})-1$.
\end{proof}

\begin{Remark}
    The rank of $X_{\mathcal{O'}}\cap p^{-1}(\mathcal{O})$ is at least the rank of $\mathfrak{b}\cap \mathfrak{h}$ over $\mathcal{O}'$, which is $\dim(H)-\dim(\mathcal{O}')$. In particular if $\dim(\mathcal{O}')=\dim(\mathcal{O})+1$ then the codimension is exactly one.
\end{Remark}
\begin{Remark}
For any simple root $\alpha$ we also have a similar construction to the one in Proposition \ref{orbit limit} for $H$ orbits on $\mathcal{P}_\alpha=P_\alpha\backslash G$. 
    
\end{Remark}

\subsection{Actions of Simple Reflections}

In this Subsection we calculate the action of a simple reflection in $W$ on a basis element of $H^{BM}_{top}(\mathfrak{h}\times_{\mathfrak{g}}\Tilde{\mathfrak{g}})$.

First, we prove several general results. After that, we do a case by case examination according to the type of an $H$ orbit on $\B$ with respect to the simple reflection. 

For the rest of this section we denote by $\s,\mathcal{T}$ two $B$ orbits on $G/H$. We sometimes identify them with $H$ orbits on $\B$. 

Let $P_s$ be the parabolic corresponding to $s$ that contains $B$, we denote by $P_s\s$ the union of $B$ orbits on $G/H$ obtained by acting with $P_s$ on $\s$, and identify them with a union of $H$ orbits on $B\backslash G$.

We write $[\mathcal{O}]$ for the class in the top Borel Moore homology defined by $X_{\mathcal{O}}$.

Consider $\mathfrak{h}\times_{\mathfrak{g}}\Tilde{\mathfrak{g}}\times_{\mathfrak{g}}\Tilde{\mathfrak{g}}$ with the three natural projections $\pi_0,\pi_1,\pi_2$ from it. The projection $\pi_0$ goes to $St'$ and the projections $\pi_1,\pi_2$ to $\Lambda'$.

For $w\in W$ and $\s$ an orbit we denote by $w\cdot [\s]$ the result of $St'_w*X_\s=\pi_{1*}(\pi_0^*St'_w\cap \pi_2^*X_\s)$ in $H^{BM}_{top}(\mathfrak{h}\times_{\mathfrak{g}}\Tilde{\mathfrak{g}})$.

Let $\alpha$ be a simple root and let $s=s_\alpha\in W$ be the simple reflection corresponding to $\alpha$.

We denote the parabolic $P_\alpha$ also by $P_s=P_{s_\alpha}$. We also denote $\mathfrak{p_\alpha}$ by $\mathfrak{p_s}$

In the following, we denote by $\mathfrak{b}_i$ the Lie algebra of a Borel subgroup $B_i$.

\subsubsection{Support of $s \cdot [\mathcal{O}]$}

In this subsubsection we bound the support of $s\cdot [\mathcal{O}]$.

\begin{prop}\label{prop: contained in closure}
    If $[\mathcal{O}']$ appears with a non-zero coefficient in $s \cdot [\mathcal{O}]$ then $\mathcal{O}'\subseteq \overline{P_{s} \mathcal{O}}$ 
\end{prop}
\begin{proof}
    $X_{\mathcal{O}}$ is supported over $\overline{\s}$. The restriction of $St'_s$ to $\overline{\mathcal{O}}$ in the first coordinate, is supported over $\overline{P_{s}\mathcal{O}}$ in the second coordinate. The result follows.
\end{proof}

For the rest of this subsection we assume $\mathcal{O}$ is not dense in $P_{s}\mathcal{O}$.

Let $\mathcal{T}$ be an orbit, we want to calculate the coefficient of $[\mathcal{T}]$ in ${s}\cdot [\mathcal{O}]$. We are looking at the intersection:
    $$ I_{\s,s}=\{(h,B_{1},B_{2})|(\mathfrak{h},B_{1})\in X_{\mathcal{O}}, (h,B_{1},B_{2})\in St'_{s}\}$$
    
If the coefficient of $[\mathcal{T}]$ in $s\cdot[\s]$ is nonzero then for a generic point $(h,B'_{2})\in X_{\mathcal{T}}$ the fiber of $I_{\s,s}$ over $(h,B'_{2})$ must be nonempty. Meaning that, for a generic $(h,B'_{2})\in X_{\mathcal{T}}$ there is $B'_1\in \B$ such that $(h,B'_1,B'_2)\in I_{\s,s}$.

\begin{lemma}
    If $\mathcal{T}$ is not dense in $P_{s}\mathcal{T}$ and is not in $P_{s}\mathcal{O}$ then the coefficient of $\mathcal{T}$ in ${s}\cdot  [\mathcal{O}]$ is zero.
\end{lemma}

\begin{proof}

    Let $\mathcal{T}$ be an orbit as in the statement of the lemma. We show that a generic point $(h,B'_{2})\in X_{\mathcal{T}}$ does not have a fiber $(h,B'_{1},B'_{2})\in I_{\s,s}$.
    
    We can extend our construction of bundles from $\B$ to $\B\times \B$. We have the constant bundle $\mathfrak{h}$ on $\B\times \B$, the bundle $\mathfrak{b_1}$ which assigns to the point $(B_1,B_2)$ the Lie algebra of $B_1$, and the bundle $\mathfrak{b_2}$ which assigns to the point $(B_1,B_2)$ the Lie algebra of $B_2$. Also, for any simple reflection $s$ we have the bundles $\mathfrak{p}_s(\mathfrak{b}_1)$ and $\mathfrak{p}_s(\mathfrak{b}_2)$ on $\B\times \B$ pulled back from $P_s\backslash G\times \B$, and $\B\times P_s\backslash G$ respectively. Furthermore, we can extend the construction of Proposition \ref{orbit limit} to this setting. 

    Let $Y_{s}=\overline{\{(B_1,B_2)\in \B\times\B|B_1\sim^sB_2\}} \subseteq \B\times \B$.
    
    Since $\mathcal{T}$ is not dense in $P_{s} \mathcal{T}$, by Corollary \ref{cor: non dense} restricting to $\mathcal{T}$ we have:

    $$ (\mathfrak{h}\cap \mathfrak{b})|_{ \mathcal{T}}=(\mathfrak{h}\cap \mathfrak{p}_{s}(\mathfrak{b}))|_{\mathcal{T}}$$
    Similarly we have:
    $$ (\mathfrak{h}\cap \mathfrak{b_{1}})|_{(\mathcal{O}\times \B) \cap Y_s}=(\mathfrak{h}\cap \mathfrak{p}_{s}(\mathfrak{b_{1}}))|_{(\mathcal{O}\times \B) \cap Y_s} = (\mathfrak{h}\cap \mathfrak{p}_{s}(\mathfrak{b_{2}}))|_{(\mathcal{O}\times \B) \cap Y_s} $$

    Notice that this bundle is the restriction of the bundle $\mathfrak{h}\cap \mathfrak{p}_{s}(\mathfrak{b_{2}})$ on $\B \times P_s \mathcal{O} $.

    Thus, for an orbit $\mathcal{T}' \subseteq P_{s}\mathcal{T} \cap \overline{\mathcal{O}}$, over $(\mathcal{T}'\times\B)\cap Y_s$ we have:

    $$ (\mathfrak{h} \cap \mathfrak{b_{1}})^{(\mathcal{O}\times \B) \cap Y_{s}}_{(\mathcal{T}'\times \B) \cap Y_{s}} \subseteq (\mathfrak{h}\cap\mathfrak{p}_{s}(\mathfrak{b_{2}}))^{(\B \times P_{s} \mathcal{O})\cap Y_{s}}_{(\B \times P_{s}\mathcal{T})\cap Y_{s}}$$

    Furthermore, for $(B_{1},B_{2})\in (\mathcal{T}' \times \B) \cap Y_{s}$:

    $$(\mathfrak{h} \cap \mathfrak{b_{1}})^{(\mathcal{O}\times \B) \cap Y_{s}}_{(\mathcal{T}'\times \B) \cap Y_{s},B_{1}} = (\mathfrak{h} \cap \mathfrak{b})^{\mathcal{O}}_{\mathcal{T}',B_{1}}$$

    $$ (\mathfrak{h}\cap\mathfrak{p}_{s}(\mathfrak{b_{2}}))^{(\B \times P_{s} \mathcal{O})\cap Y_{s}}_{(\B \times P_{s}\mathcal{T})\cap Y_{s},B_{2}}=(\mathfrak{h}\cap\mathfrak{p}_{s}(\mathfrak{b}))^{P_{s} \mathcal{O}}_{ P_{s}\mathcal{T},B_{2}}$$

    For every $(h,B'_{1},B'_{2})\in I_{\s,s}$ over $(h,B_{2}')\in X_{\mathcal{T}}$, we have $B'_{1}\in \mathcal{T}'$ for some $\mathcal{T}' \subseteq P_{s}\mathcal{T} \cap \overline{\mathcal{O}}$, and then: 
    $$h\in (\mathfrak{h}\cap \mathfrak{b})^{\mathcal{O}}_{ \mathcal{T}',B'_{1}}\subseteq(\mathfrak{h}\cap \mathfrak{p}_{s}(\mathfrak{b}))^{P_{s}\mathcal{O}}_{ P_{s}\mathcal{T}, B'_{2}}$$

    By Proposition \ref{orbit limit} the space $(\mathfrak{h}\cap \mathfrak{p}_{s}(\mathfrak{b}))^{P_{s}\mathcal{O}}_{ P_{s}\mathcal{T}, B'_{2}}$ is strictly smaller then $(\mathfrak{h}\cap \mathfrak{p}_{s}(\mathfrak{b}))_{B'_{2}}=(\mathfrak{h}\cap \mathfrak{b})_{B'_{2}}$. 
    
    We see that the fiber over a generic $(h,B_2')\in X_{\mathcal{T}}$ is empty. Thus, the coefficient of $[\mathcal{T}]$ is zero.

\end{proof}

\begin{lemma}
    If $\mathcal{T}\subset \overline{\mathcal{O}}$ then the coefficient of $[\mathcal{T}]$ in ${s}\cdot [\mathcal{O}]$ is zero.
\end{lemma}

\begin{proof}

    Let $\mathcal{T}$ be an orbit as in the statement of the lemma. We show that a generic point $(h,B'_{2})\in X_{\mathcal{T}}$ does not have a fiber $(h,B'_{1},B'_{2})\in I_{\s,s}$.

    Again we know:

    $$ (\mathfrak{h}\cap \mathfrak{b})_{\mathcal{O}}=(\mathfrak{h}\cap \mathfrak{p}_{s}(\mathfrak{b}))_{\mathcal{O}} $$

    This implies like before that for any orbit $\mathcal{T}'\subseteq P_{s}\mathcal{T} \cap \overline{\mathcal{O}}$:

    $$ (\mathfrak{h} \cap \mathfrak{b})^{\mathcal{O}}_{\mathcal{T}'}=(\mathfrak{h} \cap \mathfrak{p}_{s}(\mathfrak{b}))^{P_{s}\mathcal{O}}_{ P_{s}\mathcal{T}}|_{\mathcal{T}'}$$

    The vector bundle $(\mathfrak{h} \cap \mathfrak{p}_{s}(\mathfrak{b}))^{P_{s}\mathcal{O}}_{P_{s}\mathcal{T}}$ is pulled back from $P_{s}\backslash G$. Since $\mathcal{T}\subset \overline{\mathcal{O}}$, for any $B'_{1}\in \mathcal{T}', B'_{2}\in \mathcal{T}$ such that $(B_1',B_2')\in Y_s$ we have:
    $$ (\mathfrak{h}\cap\mathfrak{b})^{\mathcal{O}}_{\mathcal{T}',B'_{1}}=(\mathfrak{h} \cap \mathfrak{p}_{s}(\mathfrak{b}))^{P_{s}\mathcal{O}}_{ P_{s}\mathcal{T},B_1'}=(\mathfrak{h}\cap\mathfrak{b})^{\mathcal{O}}_{\mathcal{T},B'_{2}}$$
    
    Given $(h,B_{1}',B_{2}')\in I_{\s,s}$ over $(h,B_{2}')\in X_{\mathcal{T}}$, with $B_{1}'\in \mathcal{T}'$ we have

     $$ h \in (\mathfrak{h} \cap \mathfrak{b})^{\mathcal{O}}_{\mathcal{T}',B'_{1}}= (\mathfrak{h} \cap \mathfrak{b})^{\mathcal{O}}_{\mathcal{T},B'_{2}}\subsetneq (\mathfrak{h} \cap \mathfrak{b})_{\mathcal{T},B_{2}}$$

    So again the generic fiber is empty, and we are done.
    
\end{proof}

To finish the calculation we need the following lemma.

\begin{lemma}\label{lemma:diagram}
    Let $\mathcal{T},\s'$ be two $H$ orbits on $\B=B\backslash G$ such that $\mathcal{T}\subset\overline{\s'}$. Assume that $\mathcal{T}$ is dense in $P_s\mathcal{T}$ and $\s'$ is dense in $P_s\s'$. 
    For $B'\in \mathcal{T}$ we have $(\mathfrak{h}\cap \mathfrak{b})_{\mathcal{T},B'}\neq (\mathfrak{h}\cap \mathfrak{p}_{s}(\mathfrak{b}))^{\mathcal{O}'}_{\mathcal{T},B'}$.
\end{lemma}

\begin{proof}
     Consider the following commutative diagram:

    \[\begin{tikzcd}
	   {(\mathfrak{h}\cap \mathfrak{b} )^{\mathcal{O}'}_{ \mathcal{T},B'}} & {(\mathfrak{h}\cap \mathfrak{b})_{\mathcal{T},B'}} \\
	   {(\mathfrak{h}\cap \mathfrak{p}_{s}(\mathfrak{b}))^{\mathcal{O}'}_{\mathcal{T},B'}} & {(\mathfrak{h}\cap \mathfrak{p}_{s}(\mathfrak{b}))_{\mathcal{T},B'}}
	   \arrow[from=1-1, to=1-2]
	   \arrow[from=1-1, to=2-1]
	   \arrow[from=1-2, to=2-2]
	   \arrow[from=2-1, to=2-2]
    \end{tikzcd}\]

    By Proposition \ref{orbit limit} the horizontal maps are strict inclusions. The right vertical map is also a strict inclusion.
    
    As bundles on $\B$ the subbundle $\mathfrak{b}$ inside $\mathfrak{p_s}(\mathfrak{b})$ is defined by a single equation, we denote it by $\mathfrak{g}_s^*$.
    
    The right vertical inclusion in our diagram is defined by the equation $\mathfrak{g}_{s}^*$. The left vertical map is also an inclusion defined by the same equation (although it may not be a strict inclusion).

 If $(\mathfrak{h}\cap \mathfrak{b})_{\mathcal{T},B'}= (\mathfrak{h}\cap \mathfrak{p}_{s}(\mathfrak{b}))^{\mathcal{O}'}_{\mathcal{T},B'}$ then both the right vertical map and the upper horizontal map are inclusions defined by the equation $\mathfrak{g}_{s}^*$. This means that the upper horizontal map is an equality, this is a contradiction.
\end{proof}

\begin{prop}\label{prop:support}
    If $\mathcal{T}$ is not in $P_{s}\mathcal{O}$ then the coefficient of $[\mathcal{T}]$ in ${s}\cdot [\mathcal{O}]$ is zero. 
\end{prop}
\begin{proof}
    Given the previous lemmas, we only need to consider the case where $\mathcal{T}$ is dense in $P_{s}\mathcal{T}$, and $\mathcal{T} \subset\overline{\mathcal{O}'}$ for $\mathcal{O}'$ the dense orbit in $P_{s}\mathcal{O}$, but $\mathcal{T}\nsubseteq \overline{\mathcal{O}}$.

    Assume that the coefficient of $[\mathcal{T}]$ in ${s}\cdot [\mathcal{O}]$ is not zero. 

    In this case, we must have an orbit $\mathcal{T}'\subseteq P_{s}\mathcal{T}$ that is not dense in $P_s\mathcal{T}$, such that $\mathcal{T}'\subset \overline{\mathcal{O}}$.

    Given $(h,B'_{2})\in \mathring{X_{\mathcal{T}}}$ there are at most two Borel subgroups $B'_{1}$ such that $B'_{1},B'_{2}$ are in relative position $s$ and $B'_{1}\in \overline{\mathcal{O}}$. Let $\mathcal{T}'$ be the orbit of $B'_{1}$, for at least one $B_1'$ as such we have:

    $$(\mathfrak{h}\cap \mathfrak{b})^{\mathcal{O}}_{ \mathcal{T'},B'_1}\cap Lie(B_2')=\mathfrak{h}\cap Lie(B_2')$$

    As $\s$ is not dense in $P_s\s$ we also have:

       $$ (\mathfrak{h} \cap \mathfrak{b})^{\mathcal{O}}_{\mathcal{T}',B'_{1}}=(\mathfrak{h} \cap \mathfrak{p}_{s}(\mathfrak{b}))^{P_{s}\mathcal{O}}_{ P_{s}\mathcal{T},B'_{1}}=(\mathfrak{h} \cap \mathfrak{p}_{s}(\mathfrak{b}))^{\mathcal{O}'}_{ \mathcal{T},B'_{2}}$$

    A necessary condition for the coefficient of $[\mathcal{T}]$ in $s\cdot [\s]$ to be nonzero is that for generic $(h,B_2')\in X_{\mathcal{T}}$ the equality
 $(\mathfrak{h}\cap \mathfrak{p}_{s}(\mathfrak{b}))^{\mathcal{O}'}_{\mathcal{T},B_{2}'}=(\mathfrak{h}\cap \mathfrak{b})_{\mathcal{T},B_{2}'}$ holds.

   By Lemma \ref{lemma:diagram} this is a contradiction.
    
\end{proof}

\subsubsection{Calculation of $s\cdot [\mathcal{O}]$}

In the previous subsubsection we showed that the only orbits contributing to $s\cdot [\mathcal{O}]$ are the ones in $P_s\s$. In this subsubsection we compute the coefficients of those orbits.

We begin with a corollary of Proposition \ref{prop:support}. By Remark \ref{remark: intersection then convolution} we can compute the convolution $s\cdot  [\mathcal{O}]=\pi_{1*}(\pi_0^*[St'_s]\cap \pi_2^*[{X}_\s])$ by first computing the coefficient of each fundamental class of an irreducible component of $\pi_0^*St'_s\cap \pi_2^*{X}_\s$ and then pushing these fundamental classes with $\pi_{1,*}$.  

\begin{cor}
    Let $\s$ be an orbit that is not dense in $P_s\s$.  Only irreducible components of $(\pi_0^*St'_s\cap \pi_2^*{X}_\s)$ contained in the closure of $\pi_0^*St'_s\cap \pi_2^*\mathring{X}_\s$ contribute to $s\cdot [\s]$.
\end{cor}

\begin{proof}
     Let $Z$ be an irreducible component of $\pi_0^*St'_s\cap \pi_2^*{X}_\s$ that is not contained in the closure of $\pi_0^*St'_s\cap \pi_2^*\mathring{X}_\s$. Let $(h,B_1',B_2')\in Z$. The Borel $B_2'$ cannot be in $\s$ and thus $B_1'\notin P_s\s$. By Proposition \ref{prop:support} such irreducible components do not contribute to $s\cdot [\s]$. 
\end{proof}

Notice that both $\mathring{X}_\s$ and  $St'_s$ are smooth. For $\mathring{X}_\s$ this is immediate, $St'_s$ is smooth as it is the total space of a vector bundle over the smooth variety $Y_s$. $Y_s$ is smooth as it is a $\p^1$ fiber bundle over $P_s\backslash G$. In particular, $\pi_0^*St'_s, \pi_2^*\mathring{X}_\s$ are also smooth.

Proposition \ref{compact description of convolution} gives an alternative way to compute this convolution. 

 Consider $\Tilde{\mathfrak{g}}\times \B$ and let $p_1$ be the projection onto the first coordinate. Let $p_2:\Tilde{\mathfrak{g}}\times \B\rightarrow \mathfrak{g}\times \B$ be the map that is the identity on the second coordinate and $\phi:\Tilde{\mathfrak{g}}\rightarrow \mathfrak{g}$ on the first coordinate.
 We have $\pi_{1*}(\pi_0^*[St'_s]\cap \pi_2^*[{X}_\s])=p_{2*}(p_1^*[{X}_\s]\cap [St'_s])$.

\begin{prop}\label{transversal_claim}
    If $\mathcal{O}$ is not dense in $P_{s} \mathcal{O}$, then 
    the intersection of $p_{1}^*\mathring{X}_{\mathcal{O}}$ and $St'_{s}$ at a point $(h,B_1,B_2)$ is transversal inside $\Tilde{\mathfrak{g}}\times \B$ if $B_2$ is in the open orbit in $P_s\s$.

\end{prop}
\begin{proof}
    Let $\mathfrak{b}_1,\mathfrak{b}_2$ be the Lie algebras of $B_1,B_2$ respectively.

     The tangent space of $p_1^*\mathring{X}_{\mathcal{O}}$ at $(h,B_1,B_2)$ is $$\mathfrak{h}\cap \mathfrak{b_{1}}\oplus (\mathfrak{h}+\mathfrak{b_{1}})/\mathfrak{b_{1}}\oplus \mathfrak{g}/\mathfrak{b_{2}}$$
    
    For $St'_{s}$ the tangent space is 
    $$\mathfrak{b}_{1} \cap \mathfrak{b}_{2} \oplus \Delta(\mathfrak{g}) $$
    Here $\Delta:\mathfrak{g}\rightarrow \mathfrak{g}/\mathfrak{b}_{1} \oplus \mathfrak{g}/\mathfrak{b}_{2}$ is the diagonal.
    We wish to show that their sum is $$\mathfrak{b}_{1}\oplus \mathfrak{g}/\mathfrak{b}_{1} \oplus \mathfrak{g}/\mathfrak{b}_{2}$$
    The third coordinate is exhausted by $p_1^*\mathring{X}_{\mathcal{O}}$, the second one is exhausted (given the third one) by $St'_{s}$, it remains to be show that:

    $$ \mathfrak{b}_{2}\cap \mathfrak{b}_{1} + \mathfrak{h} \cap \mathfrak{b}_{1}=\mathfrak{b}_{1} $$

    Since $B_{1}$ and $B_{2}$ are in relative position $s$, $\mathfrak{b}_{2}\cap \mathfrak{b}_{1}$ has codimension one in $\mathfrak{b}_{1}$, it is enough to show:

    $$ \mathfrak{h} \cap \mathfrak{b}_{1} \nsubseteq \mathfrak{h} \cap\mathfrak{b}_{2}$$

    As $\mathcal{O}$ is not dense in $P_{s}\mathcal{O}$, and $B_{2}\in P_{s}\mathcal{O}$ is in the dense orbit of $P_s\mathcal{O}$ we have $\mathrm{dim}(\mathfrak{h} \cap \mathfrak{b}_{2}) < \mathrm{dim}(\mathfrak{h} \cap \mathfrak{b}_{1})$.

    
\end{proof}

In fact, the intersection is transversal also in a slightly more general situation. 

\begin{prop}\label{transversal_claim with extra steps} 
    The intersection of $p_{1}^*\mathring{X}_{\mathcal{O}}$ and $St'_{s}$ at the point $(h,B_1,B_2)$ is transversal inside $\Tilde{\mathfrak{g}}\times \B$ if one of the following conditions holds: 
    \begin{enumerate}
        \item $P_s\s=\s$ and $B_2\neq B_1$.
        \item $B_2=B_1$ and either $P_s\s=\s$ or $\mathcal{O}$ is not dense in $P_{s} \mathcal{O}$.
    \end{enumerate}
    
\end{prop}

\begin{proof}
    \begin{enumerate}
        \item Like in the proof of Proposition \ref{transversal_claim}, it is enough to check that     $ \mathfrak{h} \cap \mathfrak{b}_{1} \neq \mathfrak{h} \cap\mathfrak{b}_{2}$.

    Let $B_1\subset P_s$ be the parabolic corresponding to $s$ and let $U_s$ be the unipotent radical of $P_s$. We have $\s=B_1H=P_sH=U_sH$ because in type $G$, $U_s\backslash(H\cap P_s)=U_s\backslash P_s$.

    Let $\mathfrak{u}_s$ be the Lie algebra of $U_s$.
    
    Like in the proof of Proposition \ref{dimnesion calculation} we have $dim(\mathfrak{h}\cap \mathfrak{u}_s)=dim(\mathfrak{h}\cap p_s)-3=dim(\mathfrak{h}\cap b_1)-2$. Notice that if $B_1,B_2,B_3$ are three different Borels in $P_s$ then $\mathfrak{b}_1\cap \mathfrak{b}_2\cap \mathfrak{b}_3=u_s$. Consider $\mathfrak{h}\cap \mathfrak{b}_1\subset \mathfrak{h}\cap p_s$, it has codimension 1. The space  $\mathfrak{h}\cap \mathfrak{b}_1\cap \mathfrak{b}_2\subset \mathfrak{h}\cap \mathfrak{p}_s$ has codimension at most 2, and if  it has codimension 1 then $\mathfrak{h}\cap \mathfrak{b}_1\cap \mathfrak{b}_2\cap \mathfrak{b}_3\subset \mathfrak{h}\cap \mathfrak{p}_s$ would have codimension at most 2. But we know that it has codimension 3. Thus  $\mathfrak{h}\cap \mathfrak{b}_1\cap\mathfrak{b}_2$ has codimension 2 and in particular $\mathfrak{h}\cap \mathfrak{b}_1\neq \mathfrak{h}\cap\mathfrak{b}_2$.

    \item  We denote $B=B_1=B_2$ and $\mathfrak{b}=\mathfrak{b}_1=\mathfrak{b}_2$. Let $\alpha$ be the simple root corresponding to $s$. Recall Definition \ref{def:functionl alpha_s}, we have $\alpha_B:\mathfrak{b}\rightarrow \C$. Let $\mathfrak{b}^s$ be its kernel. By Lemma \ref{simple reflection steinberg} $\mathfrak{b}^s$ is the analogue for $\mathfrak{b}_1\cap \mathfrak{b}_2$ in the case of $\mathfrak{b}_1= \mathfrak{b}_2$. 

    Like in the proof of Proposition \ref{transversal_claim}, we need to show that $\mathfrak{b}^s+\mathfrak{h}\cap\mathfrak{b}=\mathfrak{b}$. It is enough to show $\mathfrak{h}\cap\mathfrak{b}\subsetneq \mathfrak{b}^s$ i.e. $\alpha_B$ is not trivial on $\mathfrak{h}\cap\mathfrak{b}$. We know that $U_s\backslash P_s\cap H$ is a spherical subgroup of the rank-one group $U_s\backslash P_s$. And we know from the classification of such (see \cite{knop}) that $\alpha_B$ is not trivial on the Lie algebra of $U_s\backslash B\cap H$, this implies that $\alpha_B$ is not trivial on $\mathfrak{h}\cap\mathfrak{b}$.   
        
    \end{enumerate}
\end{proof}

\begin{Remark}\label{remark: non_transversal}
    Notice that Proposition \ref{transversal_claim} and part two of Proposition \ref{transversal_claim with extra steps} cover almost all cases where $\s$ is not dense in $P_s\s$. The only remaining case is the one where $P_s\s$ contains three Borel orbits, $B_2,B_1$ are in different ones and none of them is in the open orbit. In this case the intersection of $p_{1}^*\mathring{X}_{\mathcal{O}}$ and $St'_{s}$ is in fact not transversal at $(h,B_1,B_2)$. It is still true that all the irreducible components of the intersection have the same dimension. This follows from the fact that the dimension of the tangent space at such a point $(h,B_1,B_2)$ is larger by one than in the generic case and the dimension of the fiber over $\B\times \B$ is smaller by one than in the generic case.
\end{Remark}

\begin{lemma}\label{coef open}
    Let $\mathcal{T}$ be the dense orbit of $P_{s}\mathcal{O}$, the coefficient of $\mathcal{T}$ in $s\cdot[\mathcal{O}]$ is $1$ for types $U_{1}, T_{1}$, and $2$ for type $N_{1}$.
\end{lemma}
\begin{proof}
    By Theorem \ref{positive coef} and Proposition \ref{transversal_claim}, the intersection in Borel Moore homology $p_1^*[{X}_\s]\cap [St'_s]$ restricted to $p_2^{-1}(X_\mathcal{T})$, is the fundamental class of the irreducible component of the set theoretic intersection over $p_2^{-1}(X_\mathcal{T})$.
    
    Therefore the coefficient of $[\mathcal{T}]$ in $p_{2*}(p_1^*[{X}_\s]\cap [St'_s])$ is equal to the size of the fiber over a generic point $(h,B_{2})\in \mathring{X}_{\mathcal{T}}$. There are either one or two Borel subgroups $B_{1}$ such that $B_{1}\in \mathcal{O}$ and $B_{1},B_{2}$ are in a relative position $s$. There is exactly one for types $U_{1},T_{1}$ and two for type $N_{1}$.

    We claim that any such $B_{1}$, $(h,B_{1},B_{2})\in I_{\s,s}$. We need to check that $h\in \mathfrak{h}\cap \mathfrak{b}_{1}$, i.e. $\mathfrak{h}\cap \mathfrak{b}_{2}\subseteq \mathfrak{h} \cap \mathfrak{b}_{1}$.
    Let $\mathfrak{b_1},\mathfrak{b_2},\mathfrak{p}_s$ be the Lie algebras of $B_1,B_2$ and the parabolic $P_s$ corresponding to $s$ that contains them.
    
    As we saw before:
    $$\mathfrak{h}\cap \mathfrak{b}_{2}\subseteq \mathfrak{h} \cap \mathfrak{p}_{s}=\mathfrak{h} \cap \mathfrak{b}_{1}$$
\end{proof}

\begin{lemma}\label{coef same}
    Let $\mathcal{T}=\mathcal{O}$, the coefficient of $\mathcal{T}$ in ${s}\cdot [\mathcal{O}]$ is $0$ for types $U_{1},T_{1}$ and $1$ for type $N_{1}$.
\end{lemma}
\begin{proof}
     By Theorem \ref{positive coef} and Proposition \ref{transversal_claim with extra steps}, the product $p_1^*[{X}_\s]\cap [St'_s]$ over $p_2^{-1}(X_\s)$ is the fundamental class of the corresponding irreducible component.

    Therefore, the coefficient of $[\mathcal{O}]$ in $p_{2,*}(p_1^*[{X}_\s]\cap [St'_s])$ is equal to the size of the fiber of a generic point $(h,B_{2})\in \mathring{X}_{\mathcal{O}}$. In types $U_{1}$ and $T_{1}$, the only choice of $B_{1}$ such that $(B_1,B_2)\in Y_s$ and also $B_{1}\in \mathcal{O}$ is $B_{1}=B_{2}$. 
    
    For $(h,B_{1},B_{1})\in I_{\s,s}$, by Lemma \ref{simple reflection steinberg} we have $\alpha_{B_1} (h)=0$.

    Therefore, if the generic point $(h,B_{1})\in \mathring{X}_{\mathcal{O}}$ has a nonempty fiber in $I_{\s,s}$, then $\mathfrak{h} \cap \mathfrak{b}_1\subseteq ker(\alpha_{B_1})$. By the proof of Proposition \ref{transversal_claim with extra steps} we know this is impossible.






    Now we consider type $N_{1}$. In this case there are two Borel subgroups $B_{1}$ such that $(B_1,B_2)\in Y_s$ and $B_{1}\in \s$, we have $B_{1}=B_{2}$ and another unique $B_{1}=B'$. The same argument shows that for generic $(h,B_{2})\in \mathring{X}_{\mathcal{O}}$, we have $(h,B_2,B_2)\notin I_{\s,s}$. It remains to check whether a generic point has a fiber with $B_{1}=B'\neq B_2$. 
    
    Let $\mathfrak{b_1},\mathfrak{b_2},\mathfrak{p}_s$ be the Lie algebras of $B_1,B_2$ and the parabolic $P_s$ corresponding to $s$ that contains them.
    In this case, since $B_{1}\neq B_{2}$ we only need to check whether $\mathfrak{h}\cap \mathfrak{b}_{2}\subseteq \mathfrak{h} \cap \mathfrak{b}_{1}$, this is true since:

    $$ \mathfrak{h} \cap \mathfrak{b}_{1}=\mathfrak{h} \cap \mathfrak{p}_{s}=\mathfrak{h} \cap \mathfrak{b}_{2}$$
\end{proof}

Now we turn to the final coefficient.

\begin{lemma}
    Let $\mathcal{O},s$ be of type $T_{1}$ and let $\mathcal{T}=s\mathcal{O}$, the coefficient of $[\mathcal{T}]$ in ${s}\cdot[\mathcal{O}]$ is $1$.
\end{lemma}
\begin{proof}

    Let $\s'$ be the open orbit in $P_s\s$. By Lemmas \ref{coef open} and \ref{coef same} together with Proposition \ref{prop:support} we have some $\alpha\in \Z$ such that: $$s\cdot [\mathcal{O}]=[\mathcal{O}'] +\alpha[\mathcal{T}]$$
    
    Furthermore, by Remark \ref{remark: non_transversal} and Theorem \ref{positive coef} we have $\alpha \geq 0$. We wish to prove $\alpha =1$.

    Similarly, we have some $\beta\in \Z$ such that:
    $$ s\cdot [\mathcal{T}]=[\mathcal{O}'] +\beta[\mathcal{O}]$$

    Subtracting these equations we get:

    $$ s \cdot ([\mathcal{O}]-[\mathcal{T}])= \alpha [\mathcal{T}] - \beta [\mathcal{O}]$$

    Acting by $s$ on this equation and using the fact that $s^2=1$ we get:

    $$ [\mathcal{O}]-[\mathcal{T}]= (\alpha - \beta)[\mathcal{O}'] + \alpha \beta ([\mathcal{O}] - [\mathcal{T}])$$

    It is easy to see that the only solutions to this equation are $\alpha =\beta = \pm 1$. Since we know $\alpha \geq 0$, then $\alpha =1$ and we are done.

    \end{proof}

Putting all of our results together, we calculated the action of ${s}$ on $[\mathcal{O}]$ for $\s$ an orbit which is not dense in $P_{s}\mathcal{O}$. The action of ${s}$ is of order $2$, and we use this to calculate the action in the remaining cases.
\begin{theorem}\label{computation result}
    Let $\mathcal{O}$ be an $H$ orbit on $\B=B\backslash G$, then in the notions of Definition \ref{type_rank_1_spheric} we have:

    \begin{itemize}
        \item (Type $G$ ) $s \cdot [\mathcal{O}]=[\mathcal{O}]$
        \item (Type $U1$) $s \cdot [\mathcal{O}]=[s\mathcal{O}]$
        \item (Type $U2$) $s \cdot [\mathcal{O}]=[s\mathcal{O}]$
        \item (Type $T1$) $s \cdot [\mathcal{O}]=[\mathcal{O}']+[s\mathcal{O}]$
        \item (Type $T2$) $s \cdot [\mathcal{O}]=-[\mathcal{O}]$
        \item (Type $N1$) $s \cdot [\mathcal{O}]=[\mathcal{O}']+[\mathcal{O}]$
        \item (Type $N2$) $s \cdot [\mathcal{O}]=-[\mathcal{O}]$

    \end{itemize}

    \begin{proof}
        The result for types $U1,T1,N1$ follows from the previous lemmas. Using these cases and the fact $s^2=1$, we can easily compute $s \cdot [\mathcal{O}]$ for types $U_{2},T_{2},N_{2}$.

        It remains to see what happens in type $G$. We show that if $\mathcal{T}\neq \s$ then the coefficient of $[\mathcal{T}]$ in $s\cdot [\s]$ is zero. By Proposition \ref{prop: contained in closure} we may assume $\mathcal{T}\subset \overline{\s}$.
        
        There are two cases, either $\mathcal{T}$ is closed in $P_s\mathcal{T}$, or $\mathcal{T}$ is dense in $P_s\mathcal{T}$. 
        
        In the first case we have $P_s\mathcal{T}\subset \overline{P_s\s}=\overline{\s}$. For $B'\in \mathcal{T}$, we have $$(\mathfrak{p}_s\cap \mathfrak{h})^\s_{\mathcal{T},B'}\subsetneq(\mathfrak{p}_s\cap \mathfrak{h})_{\mathcal{T},B'}=(\mathfrak{b}\cap \mathfrak{h})_{\mathcal{T},B'}$$
        
        So the fiber of a generic point $(h,B')\in X_\mathcal{T}$ has to be empty. 

        In the second case, a necessary condition for the coefficient of $[\mathcal{T}]$ in $s\cdot [\s]$ to be nonzero is that for generic $(h,B')\in X_{\mathcal{T}}$ the equality
 $(\mathfrak{h}\cap \mathfrak{p}_{s}(\mathfrak{b}))^{\mathcal{O}}_{\mathcal{T},B'}=(\mathfrak{h}\cap \mathfrak{b})_{\mathcal{T},B'}$ holds. By Lemma \ref{lemma:diagram} we know that this equality does not hold.

Therefore, there is some $\alpha\in \Z$ such that $s\cdot [\s]=\alpha[\s]$. By Theorem \ref{positive coef} and Proposition \ref{transversal_claim with extra steps}, we have $\alpha=1$.
    \end{proof}
\end{theorem}

As an immediate corollary, we get the following.

\begin{theorem}
    We have an isomorphism of $\C[W]$ modules $V_{comb}(X)\cong V_{geo}(X)$.
\end{theorem}

\begin{proof}
    This follows from Proposition \ref{Fourier isomorphism}, Theorem \ref{computation result} and the definition of $V_{comb}(X)$.
\end{proof}

\section{Application and examples}\label{s4}

In this section, we use Theorem \ref{main} to give a proof of Conjecture \ref{main conj} in two special cases. Let $G^\vee$ be the Langlands dual group of $G$. Let $M$ be a hyperspherical $G$ variety with a dual hyperspherical $G^\vee$ variety $M^\vee$ (see \cite{benzvi2024relativelanglandsduality}). Let $\Lambda=M\times_{\mathfrak{g}^*}T^*\B$ and $\Lambda^\vee=M^\vee\times_{\mathfrak{g}^{\vee *}}T^*\B^\vee$. We recall the formulation of the conjecture:

\begin{conj}\label{FGT conjecture}
\begin{enumerate}
    \item $\# Irr(\Lambda)=\# Irr(\Lambda^\vee)$.
    \item We have an isomorphism of $\C[W]$ modules $H^{BM}_{top}(\Lambda^\vee)\cong H^{BM}_{top}(\Lambda)\otimes sgn$.
\end{enumerate}
    
\end{conj}

We prove this conjecture for two cases: 

\begin{enumerate}
        \item $Hom(\C^n,\C^n)$ as a $GL_n\times GL_n$ variety is dual to $GL_n\times \C^n$ as a $(GL_n\times GL_n)^\vee\cong GL_n\times GL_n$ variety.
        \item $GL_{2n+1}/GL_{n+1}\times GL_n$ as a $GL_{2n+1}$ variety is dual to $GL_{2n+1}/Sp_{2n}$ as a $GL_{2n+1}^\vee\cong GL_{2n+1}$ variety.
    \end{enumerate}

Recall the notion of the rank of a $B$ orbit (Definition \ref{Borel rank}).

\begin{prop}\label{associated graded}
    Let $X$ be a spherical variety, consider the action of $W$ on $V_{comb}(X)$. There is a filtration $F^i(V_{comb}(X))$ spanned by orbits of rank at least $i$. Let $\Gamma(V_{comb}(X))=\bigoplus_{i} F^{i}(V_{comb}(X))/F^{i+1}(V_{comb}(X))$ be the associated graded of this filtration. Moreover, for every $B$ orbit $\s$ on $X$ and every $w\in W$ there is a number $a\in \{1,-1\}$ such that $w[\s]=a[w\s]$ in $\Gamma(V_{comb}(X))$.
\end{prop}

\begin{proof}
    It is enough to check that the action of a simple reflection can only raise the rank, this is well known and easy to see from the definitions of types (see \cite{knop}). It is also enough to check the second part for simple reflections, again this follows easily from the definitions of types (Definition \ref{type_rank_1_spheric}).
\end{proof}

\begin{subsection}{The case of $T^*(Hom(\C^n,\C^n))$ and $T^*(GL_n\times \C^n)$}
    Let $G=GL_n\times GL_n$ and let $B=B_n\times B_n$ be the Borel of upper triangular matrices in both copies of $GL_n$. 

    The group $G$ acts on $Hom(\C^n,\C^n)$ by $(g_1,g_2)(\phi)(x)=g_2\phi(g_1^{-1}x)$. This action induces an action on $T^*(Hom(\C^n,\C^n))$.

    The group $G$ also acts on $GL_n\times \C^n$. The action is given by $(g_1,g_2)(g',v)=(g_1g'g_2^t,g_2v)$. This induces an action on $T^*(GL_n\times \C^n)$.

    By \cite{benzvi2024relativelanglandsduality} these are relative Langlands dual Hamiltonian spaces.

    By Theorem \ref{main} it is enough to compare $V_{comb}(X)$ and $V_{comb}(X^\vee)$. According to Proposition 4.1.1 of \cite{FGT} the Borel orbits on $X$ and on $X^\vee$ are parametrized by rook arrangements on a board of size $n\times n$. The Weyl group of $GL_n\times GL_n$ is $S_n\times S_n$, it acts naturally on the set of rook arrangements by permuting the rows and columns. This action agrees with Knop's action on the set of Borel orbits.

    For completeness we reprove some of the results of \cite{FGT}.

    \begin{subsubsection}{Borel orbits on $Hom(\C^n,\C^n)$}

    We consider an element of $X=Hom(\C^n,\C^n)$ as a $n\times n$ matrix $a$. There are $n+1$ $G$ orbits on $X$. The orbit of a matrix is determined by its rank, this rank determines the number of rooks in the corresponding rook arrangement.

    The Borel $B$ acts as left and right multiplication by upper-triangular matrices. This allows us to perform row and column operations on $a$. Using these operations we can transform $a$ to a matrix with only ones and zeros and such that there are no two ones in the same row or in the same column. This is precisely the rook arrangement. 

    We proved the following.

    \begin{prop}
        The $B$ orbits on $X$ correspond to rook arrangements on a board of size $n\times n$.
    \end{prop}

    Using the described representatives of the orbits it is easy to deduce the following.

    \begin{prop}\label{rank compuation Hom()}
        The rank of a $B$ orbit on $X$ corresponding to a rook arrangement with $m$ rooks is $m$.
    \end{prop}

    \begin{proof}

        The rank of a Borel orbit of an element $a$ is equal to the dimension of the lattice of characters on $B$ which vanish on the stabilizer of $a$. Let $T\subset B$ be the maximal torus of diagonal matrices in both coordinates. Any character on $B$ is determined by its values on $T$. 

        Let $a$ be a matrix that contains only ones and zeros and such that there is at most one $1$ on every row and on every column. The stabilizer of $a$ inside $T$ is given by $Stab_T(a)=\{(t^1,t^2)\in T|t^1_{i,i}=t^2_{j,j} \text{ if } a_{i,j}=1\}$. From this the result follows.
    
    \end{proof}

    Let $s_1,...,s_{n-1}$ be the standard simple reflections generating the Weyl group of $GL_n$. Denote by $s^l_1,...,s^l_{n-1}$ the simple reflections in the first copy of $GL_n$ in $G$ and by $s^r_1,...,s^r_{n-1}$ the simple reflections in the second copy.

    We define an action of $W$ on rook arrangements.

    \begin{defn}
        Let $w\in W$, write it as $w=(w^l,w^r)$ for $w^l,w^r$ in the Weyl group of $GL_n$. The element $w$ acts on $R$ by permuting the rows with $w^l$ and the columns with $w^r$. We denote this action by $w(R)$.  
    \end{defn}

    \begin{prop}\label{B orbit in P Hom()}
        Let $s$ be a simple reflection of $G$, let $R$ be a rook arrangement and let $\s$ be the corresponding $B$ orbit on $X$. Let $P_s$ be the parabolic corresponding to $s$. The rook arrangements corresponding to $B$ orbits in $P_s\s$ are $\{R,s(R)\}$. 
    \end{prop}

    \begin{proof} 

        Let us assume that $s=s^r_i$, let $a_{R}$ be the matrix that has ones in the places of rooks and zeros elsewhere, this matrix represents the orbit corresponding to $R$. Let $\tilde{s}\in GL_n$ be the permutation matrix corresponding to $s_i$.

        Any matrix in the orbit corresponding to $R$ is of the form $B_na_{R}B_n$, we wish to find all the $B$ orbits of matrices in $B_na_{R}B_n \tilde{s}$. 

        Clearly it is enough to consider matrices in $a_{R}B_n\tilde{s}$.

        A matrix in $a_{R}B_n$ is a matrix whose all non-zero entries are to the right of a placed rook, and the entries of rooks are non-zero. Multiplying by $\tilde{s}$ switches the $i$ column with the $i+1$ column.

        We see that a matrix in $a_{R}B_n\tilde{s}$ has the following form:
        \begin{itemize}
            \item Entries are non-zero only to the right of rooks, with one possible exception of one entry to the left of a rook in column $i+1$.
            \item If there is a rook in column $i+1$, the corresponding entry at its position is zero.
            \item Entries at the positions of rooks of $s(R)$ are nonzero.
        \end{itemize}
        
        There are several possible cases:

        \begin{itemize}
            \item There are no rooks in columns $i,i+1$. 

            In this case any matrix in $a_{R}B_n\tilde{s}$ is clearly in $a_{R}B_n=a_{s(R)}B_n$

            \item There is a single rook in one of the columns $i,i+1$.

            If the rook is in column $i+1$, then we just have a matrix in $a_{s(R)}B_n$

            Assume that the rook is in column $i$ and row $r$. If the $r,i$ entry is zero then the matrix is in $a_{s(R)}B_n$, if this entry is nonzero, then the matrix is in $a_{s(R)}B_n$.

            \item There are two rooks in columns $i,i+1$.

            Assume the row number of the rook in column $i+1$ is larger than the row number of the rook in column $i$. In this case, after one row operation we can put our matrix in $a_{s(R)}B_n$.

            Assume the row number of the rook in column $i+1$ is smaller than the row number of the rook in column $i$. If the entry in the position of the rook in column $i$ is non-zero, after one row operation we can put the matrix into $a_{R}B$. Otherwise, the matrix is in $a_{s(R)}B$.
            
        \end{itemize}

        In all cases we showed that we can get a matrix in the orbit of $a_{s(R)}$, and we cannot get a matrix that is not in one of the orbits of $a_{R},a_{s(R)}$.

        The case where $s$ is in the copy of $GL_{n}$ acting from the left is the same.
        
    \end{proof}

    \begin{cor}\label{cor: action agrees with rooks 1}
        Let $s=s^l_i$ be a simple reflection of $G$, let $R$ be a rook arrangement and let $\s$ be the corresponding $B$ orbit on $X$. The type of $(s,\s)$ is $U$ if at least one of the $i,i+1$ rows of $R$ is nonempty. Otherwise, the type is $G$. The same holds if we replace $s^l_i$ by $s^r_i$ and rows by columns.

        Moreover, the action of $W$ on the Borel orbits agrees with its action on rook arrangements.
    \end{cor}

    \begin{proof}
        If both rows are empty, then by Proposition \ref{B orbit in P Hom()} $P_s\s$ is a single $B$ orbit, so the type is $G$. Otherwise, by Proposition \ref{B orbit in P Hom()} $P_s\s$ contains two Borel orbits, so the type of $(s,\s)$ is either $U$ or $N$. By Proposition \ref{rank compuation Hom()} the ranks of both $B$ orbits in $P_s\s$ are the same. Therefore the type is $U$.

        Since the action of the Weyl group is defined in terms of the types of orbits, the second part follows.
    \end{proof}
        
    \end{subsubsection}

    \begin{subsubsection}{Orbits on $GL_n\times \C^n$}

    An element of $X^\vee=GL_n\times \C^n$ is a pair of a matrix $g'$ and a vector $v$. There are two $G$ orbits, determined by whither $v=0$. In this subsection we take $B\subset G$ to be the Borel of upper triangular matrices in the first coordinate and lower triangular matrices in the second coordinate.

    \begin{prop}\label{Borel orbits on X^}
        The $B$ orbits on $X^\vee$ correspond to rook arrangements on a board of size $n\times n$.
    \end{prop}

    \begin{proof}
        
    The Borel orbits in the $G$ orbit with $v=0$ correspond to full rook arrangements. The rest correspond to non-full rook arrangements.

    There are $n+1$ Borel orbits on $\C^n$, represented by the zero vector and by the coordinate vectors $e_i$ which have zero everywhere but at the $i$ index, and there the entry is 1. 
    
    For $1\leq i\leq n$, let $B^t_i$ be the stabilizer of $e_i$ inside the subgroup of lower diagonal matrices in $GL_n$. The group $B^t_i$ is the subgroup of lower triangular matrices whose elements have $1$ in the $(i,i)$ entry and the rest of the $i$ column is zero. Denote by $B_{n+1}$ be the group of upper triangular matrices in $GL_n$.

    The transpose of $B_i^t$ which we denote by $B_i$ is the subgroup of upper triangular matrices whose elements have $1$ in the $(i,i)$ entry and the rest of the $i$ row is zero. 

    To describe the $B$ orbits on $X$, it is enough to describe the $B_i$ orbits on the flag variety of $GL_n$.

    To finish the proof of Proposition \ref{Borel orbits on X^} we show the following.

    \end{proof}

    \begin{prop}\label{B_i orbits}
        Let $1\leq i\leq n$, there is a bijection between the $B_i$ orbits on the flag variety of $GL_n$ and rook arrangements whose first empty column is column number $i$.
    \end{prop}

    \begin{proof}
        Let $g\in GL_n$, we think about it as a full rank $n\times n$ matrix, we act on the left by $B_{n+1}$ and from the right by $B_i$. 

        Consider the first nonzero element in the last row of $g$, let it be $g_{n,a}\neq 0$. Using the left action we can make sure that $g_{b,a}=0$ for all $1\leq b<n$. 

        There are two cases:

        \begin{enumerate}
            \item If $a\neq i$ we can use the right action of $B_i$ to make sure that also $g_{n,j}=0$ for all $j\neq a$. This way we can reduce the problem from $n,i$ to either $n-1,i$ or $n-1,i-1$ based on whether $a>i$ or $a<i$. In the rook arrangements it means that there is a rook at place $(n,a)$. We continue by induction.

            \item If $a=i$, let $b>i$ be the smallest such that $g_{n,b}\neq 0$. We can use the right action by $B_i$ to make sure that $g_{n,j}=0$ for all $j\notin \{i,b\}$. We reduce the problem to the same problem but for $n-1,b-1$. This is done by requiring that the action from the right preserves the last row whose entries are zero everywhere but at places $i,b$ and these entries are ones. In the process of this reduction, we erase column number $i$ which has zeros everywhere but at the last row, we also erase the last row. In the rook arrangements it means that the next empty column after $i$ is $b$.
        \end{enumerate}

        These cases cover all possible cases of rook arrangements where the first empty column is column number $i$. By induction we are done.

    \end{proof}
    \begin{Remark}\label{remark: representativie GL_n times V}
        From the proof of Proposition \ref{Borel orbits on X^} it follows that any $B$ orbit can be represented by a matrix whose entries are only $0$ and $1$, and a coordinate vector.

        Given a rook arrangement we construct a representative as follows. The special vector $v$ is the coordinate vector whose number is the number of the first empty column. If there are no empty columns then $v=0$.

        The matrix has entry $1$ at the position of each rook. A column or a row that contains a rook, contains only one non-zero entry.

        To finish the description of the representative, it is enough to define the representative of the empty rook arrangement. The representative of a general rook arrangement at rows and columns that do not contain any rooks agree with the representative of the empty rook arrangement of the appropriate size.
        
        The representative of the empty rook arrangement is given by the matrix:

        $$\begin{pmatrix}
            0& 0 & 0 & \cdots & 0 & 1 \\
            0 & 0 & 0 & \cdots & 1 & 1 \\
            \vdots & \vdots & \vdots & \ddots  & \vdots & \vdots\\
            0 & 1 & 1 & \cdots & 0 & 0 \\
            1 & 1 & 0 & \cdots & 0 & 0
        \end{pmatrix}$$

        For a rook arrangement $R$ we denote the corresponding matrix by $a_{R}$.
    \end{Remark}

    Using the described representatives of the orbits it is easy to deduce the following.

    \begin{prop}\label{rank computaion GL times V}
        The rank of a $B$ orbit on $X^\vee$ corresponding to a rook arrangement $R$ with $m$ rooks is $2n-m$.
    \end{prop}

    \begin{proof}
        For a full rook arrangement, using the same argument as in the proof of Proposition \ref{rank compuation Hom()} we get that the rank of the corresponding $B$ orbit is $n$.

        Assume that the rook arrangement $R$ is not full. Let $a_R,v_R$ be the representative described in Remark \ref{remark: representativie GL_n times V}. The vector $v_R$ is some coordinate vector $e_i$.

        The rank of the Borel orbit of $(a_R,v_R)$ is equal to the dimension of the lattice of characters on $B$ that vanish on the stabilizer of $a_R,v_R$. Let $T\subset B$ be the maximal torus of diagonal matrices in both coordinates. Any character on $B$ is determined by its values on $T$. 

         The codimension of the stabilizer of $a_R$ inside $T$ is $2n-m-1$. The codimension of the stabilizer of both $a_R$ and $v_R$ inside $T$ is $2n-m$.
        
    \end{proof}
    
    \begin{prop}\label{Borel orbits in Ps orbit}
        Let $s=s^l_i$ be a simple reflection of $G$, let $R$ be a rook arrangement and let $\s$ be the corresponding $B$ orbit on $X$. Let $P_s$ be the parabolic corresponding to $s$. The rook arrangements corresponding to $B$ orbits in $P_s\s$ are determined as follows:

        \begin{enumerate}
            \item If the $i$ and $i+1$ rows of $R$ are both nonempty then the answer is $R,s(R)$.
            \item If one of these rows is empty and the other one is not, we look at the column of the unique rook on these rows. Denote this column by $j$, count the number of empty columns before $j$, and the number of empty rows after $i+1$. If these numbers differ then the answer is again $R,s(R)$. If the numbers are equal then the answer is $R,s(R)$ and the arrangement obtained by removing the rook in column $j$.

            \item If both the $i$ and $i+1$ rows are empty the answer is $R$ and the two arrangements obtained from $R$ by adding a single rook, either in row $i$ or in row $i+1$, the column of the rook is $j$ such that the number of empty rows after $i$ is equal to the number of empty columns before $j$.  
        \end{enumerate}

    The same is true for $s=s^r_i$ if we replace the roles of rows and columns. 
    \end{prop}

    \begin{proof}

    First, consider the case $s(R)\neq R$, this is the same as having a rook in one of the rows that $s$ switches. In this case, it is easy to see that the action of $s$ on $a_{R}$ gives $a_{s(R)}$, implying that $P_{s}\s$ contains the orbit corresponding to $s(R)$.

    As $P_{s}\s$ contains at most $3$ $B$ orbits, it follows that if $R,R'$ are rook arrangements corresponding to orbits in $P_{s}\s$ such that $s(R)\neq R$ and $s(R')\neq R'$ then $s(R)=R'$. Thus, if $P_s\s$ contains three Borel orbitss then one of them corresponds to a rook arrangement $R$ such that $s(R)=R$. Therefore we can assume that $s(R)=R$.

    It is easy to see that since $s$ changes  either two empty rows or two empty columns, we can erase all non-empty rows and columns, and prove the result for the empty rook arrangement.

    The calculation is easily seen to be reducible to the case of $n=2$. We present the calculation for $n=2$.

        Let $a_R=\begin{pmatrix}
        0 & 1\\
        1& 1
    \end{pmatrix}$ and $v_R=\begin{pmatrix}
        1 \\0
    \end{pmatrix}$.

    For the rest of the proof we change our notation and let $B\subset GL_2$ be the group of upper triangular matrices.

    \begin{itemize}
        \item Assume that $s=s_1^l$.  In this case $P_s=GL_2\times B^t$. The action from the left by $GL_2$ does not change the special vector $v_R$ and it acts transitively on $GL_2$. Thus, we obtain all rook arrangements whose first empty column is the first column.

        \item Assume that $s=s_1^r$. In this case $P_s=B\times GL_2$. The stabilizer of $v_R$ in $P_s$ is contained in  $B\times B$. The $B\times B$ orbit of $a_R$ is the union of two $B\times B_1^t$ orbits in $GL_2$. This gives two rook arrangements, the empty one and the one with a single rook at the $(2,2)$ position. The Borel orbit corresponding to the rook arrangement with a single rook at the second row and the first column is also in $P_s\s$.
    \end{itemize}

    \end{proof}

    \begin{cor}\label{cor: action agrees with rooks 2}
        Let $s=s^l_i$ be a simple reflection of $G$, let $R$ be a rook arrangement and let $\s$ be the corresponding $B$ orbit on $X^\vee$. The type of $\s,s$ is $U$ if both of the $i,i+1$ rows of $R$ are nonempty. If both are empty, the type is $T$; otherwise, the type is determined by the same condition as in the second case of Proposition \ref{Borel orbits in Ps orbit}. The same holds if we replace $s^l_i$ by $s^r_i$ and rows by columns.

        Moreover, the action of $W$ on the Borel orbits agrees with its action on rook arrangements.
    \end{cor}

    \begin{proof}
        If both rows are empty, then by Proposition \ref{Borel orbits in Ps orbit} $P_s\s$ contains three $B$ orbits, so the type is $T$. In the other case where $P_s\s$ contains three $B$ orbits, the type is also $T$.

        If both rows are nonempty, then by Proposition \ref{Borel orbits in Ps orbit} $P_s\s$ contains two Borel orbits corresponding to $R,s(R)$ for some rook arrangement $R$. By Proposition \ref{rank computaion GL times V} these Borel orbits have the same rank and so the type is $U$. In the remaining case the type is also $U$ for the same reason.

        Since the action of the Weyl group is defined in terms of the types of orbits, the second part follows.
    \end{proof}

    \end{subsubsection}

    Now we are ready to prove Conjecture \ref{FGT conjecture} for this case.

      \begin{theorem}\label{first example}
        Conjecture \ref{FGT conjecture} holds for the relative Langlands dual pair $T^*(Hom(\C^n,\C^n))$ and $T^*(GL_n\times \C^n)$.
    \end{theorem}

    \begin{proof}
    By Proposition \ref{associated graded} and Corollaries \ref{cor: action agrees with rooks 1} and \ref{cor: action agrees with rooks 2} it is enough to compare each $W$ orbit of rook arrangements separately.

    In the case of $X=Hom(\C^n,\C^n)$ the representation of $W$ obtained by acting on Borel orbits corresponding to rook arrangements with $k$ rooks is isomorphic to the representation obtained by acting on the rook arrangements.
    
    In the case of $X^\vee=GL_n\times \C^n$ there are signs coming from the orbits of type $T$. We claim that the representation we get is precisely tensoring the representation obtained from acting rook arrangements by the sign character. To prove this, it is enough to take a single rook arrangement $R$ and to show that if $w\in W$ stabilizes it, then for any presentation of $w=s_1\cdot ...\cdot s_l$ as a product of simple reflections, the number of indices $j$ such that $s_j\cdot s_{j+1}\cdot ...\cdot s_l (R)= s_{j+1}\cdot ...\cdot s_l (R)$ is equal to $l$ modulo $2$.  This means that the number of indices $j$ such that $s_j\cdot s_{j+1}\cdot ...\cdot s_l (R)\neq s_{j+1}\cdot ...\cdot s_l (R)$ is even. 
    
    The number of simple reflection switching an empty row with a nonempty one is easily seen to be even by a chessboard coloring argument. 
    
    We claim that the number of simple reflection switching either two nonempty rows or two nonempty columns is also even. Consider the nonempty rows $i_1<i_2<...<i_k$ and the nonempty columns $j_1<j_2<...<j_k$. Define a permutation of $\{1,...,k\}$ by sending $1\leq t\leq k$ to $1\leq s\leq k$ if there is a rook at the $(i_s,j_t)$ position. 
    The simple reflection switching either two nonempty rows or two nonempty columns are the only ones changing the sign of this permutation.
    \end{proof}

\end{subsection}

\begin{subsection}{The case of $T^*(GL_{2n+1}/GL_n\times GL_{n+1})$ and $T^*(GL_{2n+1}/Sp_{2n})$}

Let $n\geq 1$ and let $G=GL_{2n+1}$. Let $X=GL_{2n+1}/GL_n\times GL_{n+1}$ and let $X^\vee=GL_{2n+1}/Sp_{2n}$ with the usual actions.

    By \cite{benzvi2024relativelanglandsduality} $T^*X$ and $T^*X^\vee$ are relative Langlands dual of each other.

Let $B$ be the Borel of upper triangular matrices in $G$. It is well known that the $B$ orbits on $X$ are parametrized by clans. 

\begin{defn}
    Let $m,k$ be two integers, a $(m,k)$ clan is an isomorphism class of sequences of length $m+k$ whose elements are either natural numbers or one of the symbols $\{+,-\}$. Each number appears either twice or not at all, and the difference between the number of $+$ symbols and the number of $-$ symbols is $m-k$. Two such sequences are isomorphic if there is an automorphism of $\N$ (as a set) sending one sequence to the other.

    In this work we only consider $(n+1,n)$ clans. Meaning sequences of length $2n+1$ with one more $+$ than $-$. We use clans to mean $(n+1,n)$ clans.

    For a Clan $C$ and $1\leq i\leq 2n+1$ we denote by $C_i$ the corresponding element. It is either a number or a $\{+,-\}$ sign.
\end{defn}

\begin{exmp}
    There are 6 clans for $n=1$ and they  are given by $$++-,+-+,-++,11+,1+1,+11$$
\end{exmp}

The Weyl group of $G$ is $S_{2n+1}$ and it acts on clans by permuting the sequence.

Let $s_1,...,s_{2n}\in W$ be the simple reflection given by switching the $i$ and $i+1$ elements.

\begin{subsubsection}{Orbits on $GL_{2n+1}/(GL_n\times GL_{n+1})$}

Let $X=GL_{2n+1}/(GL_n\times GL_{n+1})$. All the results we need in this case can be found in the literature, specifically in Section 4 of \cite{wyser2012glpxglqorbitclosures}. Therefore, we only state the needed results and we do not prove them. Note that the terminology of \cite{wyser2012glpxglqorbitclosures} differs from our own. Instead of using the notions of types $U,T,N,G$ the notions of "complex" and "non-compact imaginary" are used.

\begin{prop}
    There is a bijection between $B$ orbits on $X$ and clans.
\end{prop}

\begin{prop}\label{B orbits in Ps for GL/GL}
    The bijection between clans and $B$ orbits can be chosen such that the following holds.
    
    Let $s=s_i$ be a simple reflection and let $P_s$ be the parabolic corresponding to $s$. Let $\s$ be a $B$ orbit on $X$ and let $C$ be the clan corresponding to $\s$. The clans corresponding to the $B$ orbits in $P_s\s$ are determined as follows:
    \begin{enumerate}
        \item If one of the $i$ or $i+1$ elements of $C$ is a number and not both of them are equal, the answer is $C$ and $sC$. In this case $s,\s$ has type $U$.
        \item If $sC=C$ and the elements in positions $i,i+1$ of $C$ are either both $+$ or both $-$, then the answer is $C$. In this case $s,\s$ has type $G$.
        \item If $C_i$ and $C_{i+1}$ are $+$ and $-$ in some order, the answer is $C,sC$ and the clan obtained from $C$ by replacing $C_i$ and $C_{i+1}$ by some number which does not appear in $C$. In this case $s,\s$ has type $T$.
        \item If $sC=C$ and the elements in positions $i,i+1$ of $C$ are the same number, the answer is $C$ and both clans obtained from $C$ by replacing the elements in positions $i,i+1$ by $+-$ and $-+$. In this case $s,\s$ has type $T$.
    \end{enumerate}

    Moreover, the action of $W$ on the $B$ orbits agrees with its action on the clans.
\end{prop}
    
\end{subsubsection}

\begin{subsubsection}{Orbits on $GL_{2n+1}/Sp_{2n}$}

Let $X^\vee=GL_{2n+1}/Sp_{2n}$. Let $B\subset GL_{2n+1}$ be the group of upper triangular matrices.

\begin{prop}\label{B orbits on GL/Sp}
    There is a bijection between $B$ orbits on $X^\vee$ and clans.
\end{prop}

\begin{proof}
Consider the action of $G$ on the space of triples $(v,W,q)$ for $0\neq v\in\C^{2n+1}$ a vector, $W\subset \C^{2n+1}$ a $2n$ dimensional subspace that does not contain $v$ and $q$ a symplectic form on $W$. It is easy to see that $G$ acts transitively on this space and there is a point whose stabilizer is $Sp_{2n}$.  Thus, we need to compute the $B$ orbits on this space. 

Choosing a $2n$ dimensional vector space that does not contain $v$ is the same as choosing a functional that does not vanish on $v$. There are $2n+1$ Borel orbits on non-zero functionals. They are represented by the dual basis of the coordinate vectors $e_i$, denoted by $e_i^\vee$. Consider the space $N=\{(v,\phi)\in \C^{2n+1}\times (\C^{2n+1})^*|\phi(v)\neq 0\}$, the $B$ orbits on $N$ are represented by $\{(e_i,e_i^\vee+e_j^\vee)|1\leq j\leq i\leq 2n+1\}$. 

Let $\Tilde{B}_i^j$ be the stabilizer of the pair $e_i,e_i^\vee+e_j^\vee$ for $j\leq i$.

Let $W$ be the kernel of $e_i^\vee+e_j^\vee$.

We choose a basis of $W$ given by $e_1,\ldots,e_{j-1},e_{j}-e_{i},e_{j+1},\ldots,e_{i-1},e_{i+1},\ldots,e_n$.

This choice defines an embedding of $\Tilde{B}_i^j$ into $GL_{2n}$. 

The image of $\Tilde{B}_i^j$ in $GL_{2n}$ is a group denoted by $B_i^j$. If $i=j$, then $B^i_j$ is the subgroup of upper triangular matrices. Otherwise, $B^i_j$ is the subgroup of upper triangular matrices that have 1 at the $(j,j)$ entry, and the entries $(j,k)$ are zero for $k$ between $j+1$ and $i-1$ (if $j=i-1$ there is no extra condition). 

To illustrate, a matrix in $B_i^j$ is of the form:

$$\begin{pmatrix}
    * & \cdots & \cdots &\cdots & \cdots& \cdots& * &  \cdots & * \\
    \vdots & \ddots & & \cdots & \cdots&\cdots & \vdots & \cdots & * \\
        0 & \cdots  &  *& \cdots& \cdots &\cdots &   * &  \cdots  & * \\

    0 & \cdots & 0 & 1&0 & \cdots & 0 & \cdots & * \\

    0 & \cdots & \cdots & 0&* & \cdots & * & \cdots & * \\

    \vdots & & & &  & \ddots & \vdots &  &\vdots \\
    0 & \cdots & \cdots& \cdots & \cdots & 0 & * & \cdots & * \\
    \vdots & & & && & & \ddots & \vdots \\
    0 & \cdots & \cdots& \cdots & \cdots & \cdots & \cdots & 0 & *
\end{pmatrix}
$$

The proposition follows from the next one.

\end{proof}

\begin{prop}
    Let $j\leq i$, the number of $B^i_j$ orbits on $GL_{2n}/Sp_{2n}$ is equal to the number of clans such that the leftmost $+$ is in position $j$ and the rightmost $+$ is in position $i$.
\end{prop}

\begin{proof}
    Let $q$ be the symplectic form fixed by $Sp_{2n}$. Any element of $GL_{2n}$ is the Gram matrix of a basis $v_1,...,v_{2n}$. The Gram matrix is $M_{x,y}=q(v_x,v_y)$. The group $B^i_j$ acts on the basis $v_1,...,v_{2n}$.

    We prove the result by induction. Let $k,l$ be such that $q(v_k,v_l)\neq 0$ and such that there is no different pair $k'\leq k$ and $l'\leq l$ with $q(v_{k'},v_{l'})\neq 0$. We try to make $v_k,v_l$ perpendicular to all other vectors. For $m\neq k,l$ consider
    $v_m'=v_m-v_k\frac{q(v_l,v_m)}{q(v_m,v_l)}+v_l\frac{q(v_k,v_m)}{q(v_l,v_k)}$, and $v_k'=v_k,v_l'=v_l$. 
    
    If $k,l\neq j$ the transformation $v_1,...,v_{2n}\rightarrow v'_1,...,v'_{2n}$ is in $B^i_j$ so we can make the $v_k,v_l$ perpendicular to all other vectors. In the clan it means that the same number appears in positions $k$ and $l$. We continue by induction with the same problem for $2n-2$.

    We are left with cases of $k=j$ or $l=j$. Without loss of generality assume $k=j$. In the clan we put a $-$ at the $l$ place if $l<i$ and in the $l+1$ place if $l\geq i$. 

    We can make sure that $v_k$ is perpendicular to all vectors except $v_l$ and that $v_l$ is perpendicular to all vectors $v_m$ with $m<j$ or $m>i-1$. Let $j\leq t\leq i-1$ be minimal such that $q(v_t,v_l)\neq 0$. We can make sure that all other vectors except $v_t,v_k$ are perpendicular to $v_l$. In the clan we put a plus at the $t$ place. By considering the space of all vectors perpendicular to $v_l,v_k$ we reduce to the same problem for $2n-2$, $j$ is replaced by $t$ and $i$ is replaced either by $i-1$ or by $i$ depending on whether $l<i$ or $l\geq i$. 

    By induction we are done.
    
\end{proof}

 \begin{Remark}\label{rem:sp_represntatives}
        It will be useful to have a set of representatives for the Borel orbits. This can be achieved following the proof of Proposition \ref{B orbits on GL/Sp}.
        
        Any $B$ orbit can be represented by a triple $(v,W,q)$ where $v$ is of the form $e_i$ and $W$ is the kernel of the functional $e_i^\vee+ e_{j}^{\vee}$ for $j\leq i$. The symplectic form $q$ in the basis $e_1,\ldots,e_{j-1},e_{j}-e_{i},e_{j+1},\ldots,e_{i-1},e_{i+1},\ldots,e_n$ is given by an anti-symmetric matrix, it can be chosen to have only ones and zeros above the diagonal, such that every row and column has at most two entries which are not zero.

        Given a clan $C$ we construct a representative as follows. Let $A_{1},...,A_{t+1}$ be the positions of the $+$ signs in $C$, and let $B_{1},...,B_{t}$ be the positions of the  $-$ sign in $C$. Let $j=A_{0}$ be the position of the leftmost $+$, and $i=A_{t+1}$ be the position of the rightmost $+$. Let $r:\{1,...,2n+1\}\rightarrow \{1,...,2n\}$ be the function \[
r(a)=
\begin{cases}
a & \text{if } a< i\\
a-1 & \text{if } a\geq i
\end{cases}
\]

        We define the representative to be $(e_{i},ker(e_i^\vee+ e_{j}^{\vee}),q)$, the symplectic form $q$ is written in the basis $v_{1},\ldots,v_{2n}=e_1,\ldots,e_{j-1},e_{j}-e_{i},e_{j+1},\ldots,e_{i-1},e_{i+1},\ldots,e_n$ as follows. 
        \begin{itemize}
            \item For $1\leq a<b\leq 2n$, such that $C_a,C_b$ are numbers. If $C_b\neq C_a$ then $q(v_{r(a)},v_{r(b)})=0$, otherwise $q(v_{r(a)},v_{r(b)})=1$.

            \item For $1\leq a,b\leq t$ we have $q(v_{r(A_{a})},v_{r(A_{b})})=q(v_{r(B_{a})},v_{r(B_{b})})=0$.

            \item For $1\leq a,b\leq t$, $q(v_{r(A_{a})},v_{r(B_{b})})=0$ unless $a-b\in\{0,1\}$. For $a-b\in\{0,1\}$ we have $q(v_{r(A_{a})},v_{r(B_{b})})=1$ if $A_{a}<B_{b}$ and $q(v_{r(A_{a})},v_{r(B_{b})})=-1$ otherwise. 
        \end{itemize}

        To reformulate, the numbers are divided into pairs of two-dimensional symplectic spaces, while the $+,-$ signs correspond to a form where each $+$ or $-$ is orthogonal to everything except the two 'adjacent' places of the opposite sign. Notice that the rightmost $+$ isn't used, and does not correspond to a vector.
        
        For example, for the clan $C=--+++$ the resulting matrix is:

        $$\begin{pmatrix}
            0& 0 & 1 &  1 \\
            0 & 0 & 0  & 1 \\
            -1 & 0 &  0 & 0 \\
            -1 & -1 & 0 & 0
        \end{pmatrix}$$

        For a clan $C$ we denote this matrix by $A_{C}$.
    \end{Remark}

Using the given representatives of the orbits the following can be deduced.

\begin{prop}\label{rank GL/Sp}
    The rank of a $B$ orbit on $X^\vee$ corresponding to a clan $C$ with $m$ $+$ signs is $n+m$.
\end{prop}

\begin{proof}
         Let $(e_i,Ker(e^\vee_i+e^\vee_j),q)$ be the representative of the orbit corresponding to $C$ described in Remark \ref{rem:sp_represntatives}. 

        The rank of the Borel orbit of $(e_i,Ker(e^\vee_i+e^\vee_j),q)$ is equal to the dimension of the lattice of characters on $B$ that vanish on the stabilizer of $(e_i,Ker(e^\vee_i+e^\vee_j),q)$. Let $T\subset B$ be the maximal torus of diagonal matrices. Any character on $B$ is determined by its values on $T$. 

         Assume that $i=j$. Let $T_i\subset T$ be the stabilizer of $e_i$. It acts on $q$ in the same way as the torus of diagonal matrices in $GL_{2n}$ acts on the $2n\times 2n$ antisymmetric matrix $A_C$. Therefore the codimension of the stabilizer of $(e_i,Ker(e^\vee_i),q)$ in $T$ is $n+1$. 
         
         We are left with the case $i\neq j$. Let $T_{i,j}$ be the stabilizer of $e_i,e_j$ inside $T$. It acts on $q$ in the same way that the torus of $2n\times 2n$ diagonal matrices with $1$ at the $(j,j)$ entry acts on $A_C$. It is easy to see that the codimension of the stabilizer of $A_C$ inside $T_{i,j}$ is $n+m-1$. Thus, the codimension of the stabilizer of the representative inside $T$ is $n+m$.
\end{proof}

\begin{prop}\label{B orbits in Ps for GL/Sp}
    The bijection between clans and $B$ orbits constructed in the proof of Proposition \ref{B orbits on GL/Sp} satisfies the following.
    
    Let $s=s_i$ be a simple reflection and let $P_s$ be the parabolic corresponding to $s$. Let $\s$ be a $B$ orbit on $X$ and let $C$ be the clan corresponding to $\s$. The clans corresponding to the $B$ orbits in $P_s\s$ are determined as follows:
    \begin{enumerate}
        \item If both of the $C_i$ and $C_{i+1}$ are numbers and $C_i\neq C_{i+1}$, the answer is $C$ and $sC$. In this case $s,\s$ has type $U$.
        \item If $sC=C$ and $C_i=C_{i+1}$ is a number, the answer is $C$. In this case $s,\s$ has type $G$.
        \item Assume $sC=C$ and $C_i=C_{i+1}=+$. We use the notions of Remark \ref{rem:sp_represntatives}. We have $i=A_{k},i+1=A_{k+1}$ for some $k$. The answer is $C$ together with both clans obtained from $C$ by replacing the $-$ at position $B_{k}$ and one of the pluses at positions $i,i+1$ by a number which does not appear in $C$. In the case where $C_i=C_{i+1}=-$ we get a similar answer with the roles of $+$,$-$ reversed, one of the $-$ signs and the $+$ at position $A_{k+1}$ is replaced by a number.
        
        In this case $s,\s$ has type $T$.

        \item If one of the elements in positions $i,i+1$ is a number and the other one is in $\{+,-\}$ the answer is $C,sC$ unless $C$ can be obtained from another clan by the procedure described in the previous item. In that case there is an additional $B$ orbit in $P_s\s$. In the first case the type of $s,\s$ is $U$ and in the second case it is $T$.
        
        \item If $C_i$ and $C_{i+1}$ are $+$ and $-$ in some order, the answer is $C,sC$. In this case $s,\s$ has type $U$.
        
    \end{enumerate}

    Moreover, the action of $W$ on the $B$ orbits agrees with its action on the clans.
\end{prop}
\begin{proof}
        First, using the representatives constructed in Remark \ref{rem:sp_represntatives} we see that the $B$ orbit corresponding to $sC$ is always in the $P_s\s$. Let $s=s_l$, the operation that switches $e_l,e_{l+1}$ and keeps the other basis elements in place is in $P_s$. If $sC\neq C$, it is easy to see that this operation sends the representative of $C$ to the representative of $sC$.

    Each $P_s$ orbit contains at most three Borel orbits and by Proposition \ref{rank GL/Sp} the rank of the orbits corresponding to $C,sC$ are the same. Therefore, like in the proof of Proposition \ref{Borel orbits in Ps orbit} it is enough to prove the result under the assumption $sC=C$.

    The case where $s$ switches two copies of the same number is easy, it is the same as what happens in the well studied case of $GL_{2n}/Sp_{2n}$.

    The interesting case is where $s$ switches two adjacent identical signs. 
    
    We begin with the case where $s$ switches two adjacent $-$.

    Let $j\leq i$ be the indices of the leftmost $+$ sign and rightmost $+$ sign respectively. We know that $i\neq j$ as there are at least two $-$ signs in $C$ and thus at least three $+$ signs. Let $s=s_l$ the reflection that switches the $l$ and $l+1$ positions. Since $s$ switches two $-$ signs, we get $l\neq i,j$ and $l+1\neq i,j$. Let $B^{j}_{i}$ be the subgroup of $GL_{2n}$ defined in the proof of Proposition \ref{B orbits on GL/Sp}. It is the subgroup of upper triangular matrices that have 1 at the $(j,j)$ place and the entries $(j,k)$ are zero for $k$ between $j+1$ and $i-1$ (if $j=i-1$ there is no extra condition). 
    
    Recall the function $r$ from Remark \ref{rem:sp_represntatives}, $r:\{1,...,2n+1\}\rightarrow \{1,...,2n\}$ \[
r(a)=
\begin{cases}
a & \text{if } a< i\\
a-1 & \text{if } a\geq i
\end{cases}
\]

    Let $P^{j}_{i}$ be the group generated by $B^{j}_{i}$ and the $s_{r(l)}$ permutation matrix, it consists of matrices of the same form as $B_{j}^{i}$ but allowing also a non-zero entry at the $(r(l)+1,r(l))$-coordinate. We wish to calculate the $B_{i}^{j}$ orbits in $P^{j}_{i}A_{C}$. In fact, it is enough to show that a representative of the clan obtained by the process in the statement of the proposition is in $P^{j}_{i}A_{C}$, since we know there are at most three $B_{i}^{j}$ orbits in $P^{j}_{i}A_{C}$.

    Consider $A_{C}$ as the Gram matrix of a basis $v_{1},...,v_{2n}$. We use the notions of Remark $\ref{rem:sp_represntatives}$. We know that for some $k$, $l=B_{k},l+1=B_{k+1}$ and that $v_{r(A_{k+1})}$ is orthogonal to all but $v_{r(l)},v_{r(l+1)}$.

    Acting by $P^{j}_{i}$ we may replace $v_{r(l)}$ by $v_{r(l)}-v_{r(l)+1}$. After this operation $v_{r(l)}$ and  $v_{r(A_{k+1})}$ become orthogonal. Acting by $B^{j}_{i}$ we can replaces $v_{r(A_{k+2})}$ by a linear combination of $v_{r(A_{k+2})},v_{r(A_{k+1})}$ that is orthogonal to $v_{r(l)+1}$ (if $A_{k+2}$ isn't defined this step is not needed). After these operation $v_{r(l)+1}$ and $v_{r(A_{k+1})}$ are orthogonal to all other vectors. Up to multiplication by a scalar, the Gram matrix of this new basis is exactly $A_{D}$, for the clan $D$ obtained from $C$ by replacing the $B_{k+1},A_{k+1}$ elements with a new number. 

    In the case where $s$ switch two $+$ signs, unless this is the rightmost pair, the argument remains the same.

    We are left with the case where $s$ switches the two rightmost $+$ signs. i.e. $l=i-1$. We again use the notions of Remark \ref{rem:sp_represntatives}. Acting by the permutation matrix of $s=s_{i-1}$, we get the element $(e_{i-1},ker(e_{i-1}^{\vee}+e_{j}^{\vee}),q)$, where $q$ is the form obtained from $A_{C}$ and the linear isomorphism between $ker(e_{i-1}^{\vee}+e_{j}^{\vee})$ and $ker(e_{i}^{\vee}+e_{j}^{\vee})$ induced by $s$.

    Working in the basis $v_{1},\ldots ,v_{2n} = e_{1},e_{2}, \ldots ,e_{j-1},e_{i-1}-e_{j},e_{j+1},\ldots ,e_{i-2},e_{i}, \ldots e_{2n+1}$, we can see that $q$ is almost of the desired form.

    Replacing $v_{i}$ by $v_{i}\pm v_{r(A_{t-1})}\pm v_{r(A_{t-2})} \pm \cdots \pm v_{r(A_{0})} $ for appropriately chosen signs, we can make sure that $v_{i}$ is orthogonal to all vectors except $v_{r(B_{t})}$. Up to multiplication by a scalar this operation gives the representative for the clan $D$, obtained by changing the rightmost $+$ and rightmost $-$ into a new number. 

Again, the action of $W$ on the Borel orbits is determined by the types, so the fact that it agrees with the action on clans follows
immediately.
\end{proof}

\end{subsubsection}
    
Now we are ready to prove Conjecture \ref{FGT conjecture} for this case.

      \begin{theorem}
        Conjecture \ref{FGT conjecture} holds for the relative Langlands dual pair $T^*(GL_{2n+1}/GL_n\times GL_{n+1})$ and $T^*(GL_{2n+1}/Sp_{2n})$.
    \end{theorem}

    \begin{proof}
    By Propositions \ref{associated graded}, \ref{B orbits in Ps for GL/GL} and \ref{B orbits in Ps for GL/Sp} it is enough to compare each $W$ orbit of clans separately.

    Notice that if a simple reflection $s$ fixes a clan $C$ then the corresponding type is always either $G$ or $T2$, the type of $C,s$ is $G$ for $X$ if and only if the type of $C,s$ is $T$ for $X^\vee$.

    Thus, like in the proof of Theorem \ref{first example} it is enough to show that if $w\in W$ fixes a clan $C$ and $w=s_1\cdot ... \cdot s_l$ is a presentation of $w$ as a product of simple reflections. Then the number of indices $j$ such that $$s_j\cdot s_{j+1}\cdot ...\cdot s_l C\neq s_{j+1}\cdot ...\cdot s_l C$$ is even.

    By considering the parity of situations where we have two numbers $a,b$ where the first appearance of $a$ in $C$ is before the first appearance of $b$ but the second appearance of $a$ is after the second appearance of $b$, we see that the number of simple reflections switching two different numbers is even. 
    
    By considering the parity of the sum of all positions where either a $+$ or a $-$ appears we see that the number of simple reflections switching a number and a $\{+,-\}$ sign is even. 

    It remains to show that the number of simple reflections switching $+,-$ to $-,+$ is even.  Consider the parity of the number of pairs of indices $1\leq i<j\leq 2n+1$ such that $C_{j}=+$ and $C_{i}=-$. The simple reflections switching $+,-$ to $-,+$ are the only ones changing this parity.
        
\end{proof}

\end{subsection}
    
\newpage

\appendix

\section{Intersections in Borel Moore Homology}\label{A1}

In this Appendix we mention several properties of Borel Moore homology that we use throughout the paper. For more details see Sections 2.6 and 2.7 of \cite{Chriss1997RepresentationTA} and Section 19 of \cite{FultonIntersectionTheory}.

\begin{defn}
    Let $X$ be a complex algebraic variety of dimension $n$ over $\R$. Let $X^{sm}$ be the smooth points of $X$. We define a fundamental class of $X$ in $H_{top}^{BM}(X)=H_{n}^{BM}(X)$ to be a class which restricts to the fundamental class of $X^{sm}$ in $H_{top}^{BM}(X^{sm})$. 
\end{defn}

\begin{prop}
    A fundamental class exists and it is unique.
\end{prop}

\begin{proof}
    We give the sketch of the proof, for more details see 2.6.12 in \cite{Chriss1997RepresentationTA}. It is enough to show that $H_{top}^{BM}(X)\cong H_{top}^{BM}(X^{sm})$. This follows from the fact that the codimension of $X^{sm}$ is at least 2 (over $\R$) and the long exact sequence of Borel Moore homology.
\end{proof}

The fundamental class of $X$ is denoted by $[X]$.

\begin{prop}\label{basis of top BM}(Proposition 2.6.14 in \cite{Chriss1997RepresentationTA})
    Let $X$ be a complex algebraic variety, the fundamental classes of the irreducible components of $X$ of maximal dimension form a basis for $H_{top}^{BM}(X)$.
\end{prop}

\begin{defn}\label{def:intersection}
    Let $M$ be a smooth complex algebraic manifold of real dimension $m$ and let $X_1,X_2\subset M$ be two closed algebraic varieties such that there are open subsets $U_1,U_2\subset M$ with $X_1,X_2$ being homotopy retracts of $U_1,U_2$ respectively. We define $$\cap:H^{BM}_{i}(X_1)\times H^{BM}_{j}(X_2)\rightarrow H^{BM}_{i+j-m}(X_1\cap X_2)$$ using the cup product in relative cohomology $$\cup:H^{n-i}(M,M\backslash X_1)\times H^{n-j}(M,M\backslash X_2)\rightarrow H^{2m-i-j}(M,M\backslash X_1\cup X_2)$$ and Poincare duality.
\end{defn}

An important result about the intersection is the projection formula. For its proof see \cite{FultonIntersectionTheory}.

\begin{prop}(Projection formula)
    Let $f:X\rightarrow Y$ be proper. Let $x\in H_i^{BM}(X), y\in H_i^{BM}(X)$ then we have $$f_*(x\cap f^*y)=f_*x\cap y$$
\end{prop}

\begin{defn}
    Let $M,X_1,X_2$ be like in Definition \ref{def:intersection}. Let $Z$ be an irreducible component of $X_1\cap X_2$. Assume that $X_1,X_2$ are equidimensional. We say that $X_1,X_2$ intersect properly at $Z$  if $dim(Z)=dim(X_1)+dim(X_2)-dim(M)$. 
\end{defn}

\begin{theorem}\label{positive coef}
    Let $M,X_1,X_2$ be like in Definition \ref{def:intersection}. Assume that $X_1,X_2$ are equidimensional. Write $[X_1]\cap [X_2]\in H^{BM}_{top}(X_1\cap X_2)$ as a linear combination of the fundamental classes of all the top dimensional irreducible components of $X_1\cap X_2$. 
    
    \begin{enumerate}
        \item All the coefficients of the fundamental classes where $X_1,X_2$ intersect properly are positive. 
    \item  Let $Z$ be a top dimensional irreducible component of $X_1\cap X_2$. Assume that we have $\mathring{X}_1,\mathring{X}_2$ smooth and open in $X_1,X_2$ respectively, and assume that they intersect transversely at an irreducible component $\mathring{Z}\subset \mathring{X}_1\cap \mathring{X}_2$. If $Z\subset \overline{\mathring{Z}}$ then the coefficient of $[Z]$ in $[X_1]\cap [X_2]$ is 1.
    
    \end{enumerate}

\end{theorem}
\begin{proof}
    See Appendix A1 of \cite{MR2179652}, as well as Proposition 8.2 of \cite{FultonIntersectionTheory}.
\end{proof}

Next we define the convolution construction.

\begin{defn}
    Let $M_1,M_2,M_3$ be three smooth complex algebraic manifolds and let $X_1\subset M_1\times M_2$ and $X_2\subset M_2\times M_3$ be two closed algebraic varieties. Let $X_1\circ X_2$ be the set theoretic composition $X_1\circ X_2=\{(m_1,m_3)\in M_1\times M_3|\exists m_2\in M_2 (m_1,m_2)\in X_1, (m_2,m_3)\in X_2\}$. Let $p_{i,j}:M_1\times M_2\times M_3\rightarrow M_i\times M_j$ be the projection, assume that $p_{1,3}$ is proper. Let $d$ be the real dimension of $M_2$.

    We define $$*:H_i^{BM}(X_1)\times H_j^{BM}(X_2)\rightarrow H_{i+j-d}^{BM}(X_1\circ X_2)$$ by $c_1*c_2:=p_{1,3*}(p^*_{1,2}c_1\cap p^*_{2,3}c_2)$. The intersection is carried out inside the smooth variety $M_1\times M_2\times M_3$.
\end{defn}

We give an equivalent way to define the convolution. Denote $Y=\{(m_1,m_2,m_3)\in M_1\times M_2\times M_3|(m_1,m_2)\in X_1,(m_2,m_3)\in X_2\}$. Let $p'_{1,3}:Y\rightarrow X_1 \circ X_2$, $p'_{1,2}:Y\rightarrow X_1$ and $p'_{2,3}:Y\rightarrow X_2$ be the natural projections. For $(c_1,c_2)\in H_{i}^{BM}(X_1)\times H_{j}^{BM}(X_2)$ we can compute $p'_{1,3*}(p_{1,2}'^{*}c_1\cap p'^*_{2,3}c_2)\in H_{i+j-d}^{BM}(X_1)$. Here the intersection $p_{1,2}^*c_1\cap p^*_{2,3}c_2$ is computed in $H^{BM}_\bullet(Y)$. Notice that $Y$ may be non smooth. It is easy to see using the projection formula that the following holds.

\begin{prop}\label{alternative convolution}
    In the above setup we have $p'_{1,3*}(p_{1,2}'^*c_1\cap p'^*_{2,3}c_2)=c_1*c_2$.
\end{prop}

We are interested in the convolution construction, when, for $i=\dim(X_1)$ and $j=\dim(X_2)$, we have $i+j-d=dim(X_1\circ X_2)$. In this case we write 
$$*:H_{top}^{BM}(X_1)\times H_{top}^{BM}(X_2)\rightarrow H_{top}^{BM}(X_1\circ X_2)$$.

\begin{Remark}\label{remark: intersection then convolution}
    By definition the convolution construction is computed in two steps. First we compute the intersection $p^*_{1,2}c_1\cap p^*_{2,3}c_2$ in $M_1\times M_2\times M_3$ and then we push to $X_1\circ X_2$. If $dim(Y)=i+j-d$ we can write $p^*_{1,2}c_1\cap p^*_{2,3}c_2$ as a linear combination of the fundamental classes of the top-dimensional irreducible components of $Y$. We can compute the convolution construction by computing the coefficients of the fundamental classes of the irreducible components of $Y$ in $p^*_{1,2}c_1\cap p^*_{2,3}c_2$ and the pushforward of these fundamental classes to $X_1\circ X_2$. 

    We can write the convolution as a linear combination of fundamental classes of irreducible components of $X_1\circ X_2$. A necessary condition for a fundamental class of an irreducible component of $X_1\circ X_2$ to appear in the convolution with a nonzero coefficient is that the generic fiber of $Y$ above this irreducible component is nonzero.
\end{Remark}

Let us specialize the convolution construction to one of our cases of interest. Let $\mathfrak{g}$ be a Lie algebra of a reductive group and let $\B$ be the flag variety of $\mathfrak{g}$. Let $\phi:\Tilde{\mathfrak{g}}\rightarrow \mathfrak{g}$ be its Grothendieck alteration, it is a proper surjective map of smooth varieties. Recall that $\Tilde{\mathfrak{g}}=\{(g,\mathfrak{b})|\mathfrak{b}\in\B,g\in \mathfrak{b}\subset \mathfrak{g}\}$. Let $\mathfrak{h}\subset \mathfrak{g}$ be a Lie algebra of a spherical subgroup. We take $X_1=\mathfrak{h}\times_{\mathfrak{g}}\Tilde{\mathfrak{g}}$ and $X_2=\Tilde{\mathfrak{g}}\times_{\mathfrak{g}}\Tilde{\mathfrak{g}}$, we have $X_1\circ X_2=X_1$. Both $X_1$ and $X_2$ are equidimensional of dimensions $dim(\mathfrak{h})$ and $dim(\mathfrak{g})$ respectively. We get a convolution construction $:H_{i}^{BM}(X_1)\times H_{j}^{BM}(X_2)\rightarrow H_{i+j-dim(\mathfrak{g})}^{BM}(X_1)$.

We give one more description of the convolution in this setting.  Consider $\Tilde{\mathfrak{g}}\times \B$ and let $p_1$ be the projection onto the first coordinate. Let $p_2:\Tilde{\mathfrak{g}}\times \B\rightarrow \mathfrak{g}\times \B$ be the map that is the identity on the second coordinate and $\phi$ on the first coordinate. Starting with $(c_1,c_2)\in H_{i}^{BM}(X_1)\times H_{j}^{BM}(X_2)$ we can compute $p_{2*}(p_1^*c_1\cap c_2)\in H_{i+j-dim(\mathfrak{g})}^{BM}(X_1)$. Here we consider $X_1\subset \Tilde{\mathfrak{g}}$ and $X_2\subset \Tilde{\mathfrak{g}}\times \B$.

\begin{prop}\label{compact description of convolution}
    In the above setup we have $p_{2*}(p_1^*c_1\cap c_2)=c_1*c_2$.
\end{prop}

\begin{proof}

    Using the closed embedding $\mathfrak{h}\rightarrow\mathfrak{g}$, the projection formula and proper base change we may assume that $\mathfrak{h}=\mathfrak{g}$.

    Let $Y=\mathfrak{g}\times_{\mathfrak{g}}\Tilde{\mathfrak{g}}\times_{\mathfrak{g}}\Tilde{\mathfrak{g}}$ and let $p_{1,2},p_{1,3}:Y\rightarrow X_1$ and $p_{2,3}:Y\rightarrow X_1$ be the natural projections. Notice that $p_{2,3}$ is an isomorphism.

    By Proposition \ref{alternative convolution} we have $p_{1,3*}(p_{1,2}^*c_1\cap p^*_{2,3}c_2)=c_1*c_2$. The intersection is computed inside $Y$. The intersection $p_1^*c_1\cap c_2$ is computed inside $\Tilde{\mathfrak{g}}\times \B$. Let $i:Y\rightarrow \Tilde{\mathfrak{g}}\times \B$ be the closed embedding given by $i(h,\mathfrak{b}_1,\mathfrak{b}_2)=((h,\mathfrak{b}_1),\mathfrak{b}_2)$. 

    Notice that $p_{1,2}=p_1\circ i$ and so by the projection formula $i_*(p_{1,2}^*c_1\cap p^*_{2,3}c_2)=i_*(i^*p_{1}^*c_1\cap p^*_{2,3}c_2)=p^*_1c_1\cap i_*p_{2,3}^*c_2=p^*_1c_1\cap c_2$. The result follows.
\end{proof}

\bibliographystyle{alphaurl}
\bibliography{mybib}

\end{document}